%% file: main.tex
\documentclass[11pt]{article}

\usepackage{amsmath}
\usepackage{amsfonts}
\usepackage{amssymb}
\usepackage{amsthm}
\usepackage{breqn}
\usepackage{setspace}
\usepackage{fullpage}
\usepackage{enumitem}
\usepackage{bbold} 
\usepackage{comment}
\usepackage{hyperref}
\usepackage{bbm}
\usepackage{pgf,tikz}
\usepackage{subcaption}
\usetikzlibrary{backgrounds,automata}
\usepackage{mathtools} %for := \coloneqq and \eqqcolon
\usetikzlibrary{patterns,arrows,decorations.pathreplacing}
\usepackage{float}
\restylefloat{table}
\newtheorem{lemma}{Lemma}
\newtheorem{theorem}{Theorem}

\newcommand{\F}{\mathcal{F}}

\newcommand{\N}{\mathcal{N}}

\newcommand{\abs}[1]{\left\lvert{#1}\right\rvert}

\DeclareMathOperator{\ex}{ex}

 \usepackage{authblk}

\title{The Maximum Number of Pentagons in a Planar Graph}
 \author[1]{Ervin Gy\H{o}ri} 
 \author[2,3]{Addisu Paulos}
 \author[1,4]{Nika Salia}
 \author[1,5,6]{Casey Tompkins}
 \author[2,7]{Oscar Zamora} 
\affil[1]{Alfr\'ed R\'enyi Institute of Mathematics, Hungarian Academy of Sciences.}
 \affil[2]{Central European University, Budapest.}
 \affil[3]{Addis Ababa University, Addis Ababa}
 \affil[4]{Extremal Combinatorics and Probability Group, IBS, Daejeon, Republic of Korea.}
 \affil[5]{Karlsruhe Institute of Technology, Germany.}
 \affil[6]{Discrete Mathematics Group, IBS, Daejeon, Republic of Korea.}
\affil[7]{Universidad de Costa Rica, San Jos\'e.}

\begin{document}
\maketitle

\begin{abstract}
In 1979, Hakimi and Schmeichel considered the problem of maximizing the number of cycles of a given length in an $n$-vertex planar graph.  They precisely determined the maximum number of triangles and $4$-cycles and presented a conjecture for the maximum number of pentagons. In this work, we confirm their conjecture. Even more, we characterize the $n$-vertex, planar graphs with the maximum number of pentagons.
\end{abstract}

\section{Introduction}
It is well-known that an $n$-vertex planar graph can have at most $3n-6$ edges.   Given this fact, it is natural to ask about the maximum possible number of other substructures. In 1979, Hakimi and Schmeichel~\cite{hakimi} initiated such a study by determining the maximum number of triangles and $4$-cycles in an $n$-vertex planar graph (see also~\cite{ahmad}). For a given graph $H$, let $f(n,H)$ denote the maximum number of (not necessarily induced) copies of the graph $H$ in an $n$-vertex planar graph.

Alon and Caro~\cite{alon} determined the value of the function $f(n,H)$ in the case when $H$ is a complete bipartite graph with the smaller part of size at most two. In the same paper they stated a conjecture due to Perles that for any 3-connected planar graph $H$, we have $f(n,H) \le c_H n$ for some constant $c_H$ depending on $H$.  
This conjecture was subsequently verified by Wormald in~\cite{Wormald} and independently by Eppstein in~\cite{eppstein}. The exact value of $f(n,K_4)$ was determined by Wood in~\cite{k4}.

%We determined $f(n,P_3)$,  where $P_3$ is the path of length three, in~\cite{other}.  

Hakimi and Schmeichel~\cite{hakimi} proved that $f(n,C_5)\leq 5n^2-26n$.  Furthermore, they conjectured that a bound of $2n^2-10n+12$ should hold, which is attained by the graph $D_n$ obtained from a cycle on $n-2$ vertices by adding two vertices that are adjacent to each vertex of the cycle  (see Figure~\ref{con}). 
We confirm that their conjecture holds (for $n\ge 8$), and provide a complete characterization of the extremal graphs for all $n$. In the case when $n$ is equal to $8$ or $11$ there are two further extremal graphs $A_8$ and $A_{11}$, pictured in Figure~\ref{con}.
%Our main result is the following.
\begin{theorem}
\label{c5}
 For an integer $n$ such that $n=6$ or $n \ge 8$, 
 \[
 f(n,C_5)=2n^2-10n+12.
 \]
We have $f(5,C_5)=6$ and $f(7,C_5)=41$. Furthermore, for $n\notin \{8,11\}$ the unique $n$-vertex planar graph with $f(n,C_5)$ copies of $C_5$ is $D_n$. 
For $n\in \{8,11\}$ the graphs $D_n$ and $A_n$ are the only $n$-vertex planar graphs with $f(n,C_5)$ copies of $C_5$.

%that maximize the number of copies of $C_5$ are the family of graphs $D_n$ When $n=8$ or $n=11$ the graphs $A_n$ (see Figure~\ref{con}) also achieve the maximum.
\end{theorem}

\begin{figure}   
\begin{subfigure}      {0.23\linewidth}
\centering
\begin{tikzpicture}[scale=0.8] %D_n
\foreach \x in{0,1,2,3,4,5}{\draw[fill=black](0,\x)circle(3pt);};
\draw[fill=black](-2,2.5)circle(3pt);
\draw[fill=black](2,2.5)circle(3pt);
\foreach \x in{0,1,2,4}{\draw[thick](0,\x)--(0,\x+1);}
\draw (0,3.625)node{$\vdots$};
\foreach \x in{0,1,2,3,4,5}{\draw[thick](-2,2.5)--(0,\x)--(2,2.5);}
\draw (0,5) arc (90:270:2.5cm and 2.5cm);
\end{tikzpicture}
\caption{$D_n$}
\end{subfigure}
\begin{subfigure}      {0.23\linewidth}
\centering
\begin{tikzpicture}[scale=.54] %8
\filldraw (0,0) circle (5pt)  -- (30:1cm)  circle (5pt) -- (150:1cm) circle (5pt)  -- (270:1cm) circle (5pt)  (0,0) -- (150:1cm) (270:1cm)--(30:1cm) (270:1cm) -- (0,0);
\filldraw   (270:1cm) -- (210:2cm) circle (5pt)  --  (150:1cm);
\filldraw   (30:1cm) -- (330:2cm) circle (5pt)  --  (270:1cm);
\filldraw   (90:5cm) circle (5pt);
\draw[shift= {(30:1cm)},rotate=-80] (0,0) arc (0:180: 2.3cm and 1cm);
\draw[shift= {(150:1cm)},rotate=-100] (0,0) arc (0:-180: 2.3cm and 1cm);
\draw (270:1cm) arc(240:480: 2.2cm and 3.5cm);
\draw[shift= {(330:2cm)},rotate=-74] (0,0) arc (0:180: 3.13cm and 1.2cm);
\draw[shift= {(210:2cm)},rotate=-106] (0,0) arc (0:-180: 3.13cm and 1.2cm);
\filldraw   (30:1cm) -- (90:3cm) circle (5pt)  --  (150:1cm);
\filldraw   (90:5cm)  --(90:3cm);
\end{tikzpicture}
\caption{$A_8$}
\end{subfigure}
\begin{subfigure}      {0.23\linewidth}
\centering
\begin{tikzpicture}[scale=.54]%11
\filldraw (0,0) circle (5pt)  -- (30:1cm)  circle (5pt) -- (150:1cm) circle (5pt)  -- (270:1cm) circle (5pt)  (0,0) -- (150:1cm) (270:1cm)--(30:1cm) (270:1cm) -- (0,0);
\filldraw   (270:1cm) -- (210:2cm) circle (5pt)  --  (150:1cm);
\filldraw   (30:1cm) -- (330:2cm) circle (5pt)  --  (270:1cm);
\filldraw   (90:5cm) circle (5pt);
\draw[shift= {(30:1cm)},rotate=-80] (0,0) arc (0:180: 2.3cm and 1cm);
\draw[shift= {(150:1cm)},rotate=-100] (0,0) arc (0:-180: 2.3cm and 1cm);
\draw (270:1cm) arc(240:480: 2.2cm and 3.5cm);
\draw[shift= {(330:2cm)},rotate=-74] (0,0) arc (0:180: 3.13cm and 1.2cm);
\draw[shift= {(210:2cm)},rotate=-106] (0,0) arc (0:-180: 3.13cm and 1.2cm);

\filldraw   (90:5cm)  --(90:3cm);

\filldraw (90:3cm) -- (70:3cm) circle (5pt)  -- (90:5cm)  (70:3cm) -- (30:1cm);
\filldraw (90:3cm) -- (110:3cm) circle (5pt) -- (90:5cm) (110:3cm) -- (150:1cm);

\filldraw   (30:1cm) -- (90:1.5cm) circle (5pt)  --  (150:1cm);
\filldraw   (30:1cm) -- (90:3cm) circle (5pt)  --  (150:1cm) (90:1.5cm)--(90:3cm);

\end{tikzpicture}
\caption{$A_{11}$}
\end{subfigure}
\begin{subfigure}      {0.23\linewidth}
\begin{tikzpicture}[scale=0.8]
\foreach \x in{1,2,3,4,5}{\draw[fill=black](0,\x)circle(3pt);};
\draw[fill=black](-2,0)circle(3pt);
\draw[fill=black](2,0)circle(3pt);
\foreach \x in{1,2,4}{\draw[thick](0,\x)--(0,\x+1);}
\draw (0,3.625)node{$\vdots$};
\foreach \x in{1,2,3,4,5}{\draw[thick](-2,0)--(0,\x)--(2,0);}
\draw[black,thick](-2,0)--(2,0);
\end{tikzpicture}
\caption{$E_n$}
\end{subfigure}
\caption{The graphs $D_n$, $A_8$, $A_{11}$ and $E_n$.}
\label{con}
\end{figure}

 Following the initial appearance of this work several further results have been obtained by both the present group and other groups of authors. These include determining the order of magnitude of $f(n,H)$ when $H$ is a tree~\cite{genplanar} or arbitrary graph~\cite{surface}. 
 Moreover, an exact result was attained for the path $P_4$~\cite{other} as well as several asymptotic results.
 Recently Cox and Martin~\cite{coxmartin,coxmartin2} introduced some general reduction lemmas which allowed them to determine $f(n,C_{2k})$ asymptotically for $k\in\{3,4,5,6\}$.  
 Using these lemmas the asymptotic value of $f(n,C_{2k})$ was determined for all $k$~\cite{evencyc}, and good estimates for $f(n,P_{2k+1})$~\cite{shapira} have also been obtained.
 
The problem of maximizing the number of copies of a graph $H$ in a planar graph is related to some other problems which have been considered recently.  Alon and Shikhelman~\cite{AlonShikhelman}  introduced a generalized extremal function $\ex(n,H,\F)$, defined to be the maximum number of copies of $H$ possible in an $n$-vertex graph, containing no subgraph $F \in \F$.
When $\F$ contains only one graph~$F$, we write simply $\ex(n,H,F)$.  Zykov~\cite{zykov} (and independently Erd\H{o}s~\cite{erdos}) completely resolved the case when $H$ and $F$ are both cliques. 
The problem of determining $\ex(n,C_5,C_3)$ was a well-known conjecture of Erd\H{o}s and was finally settled by Hatami, Hladk\'y, Kr\'al, Norine and Razborov~\cite{HHKNR2013} and independently by Grzesik~\cite{G2012}.

By Kuratowski's~\cite{kur} theorem, the problem of determining $f(n,H)$ is equivalent to determining $\ex(n,H,\F)$, where $\F$ is either the family of $K_{3,3}$ and $K_5$ subdivisions. 
Similar problems have also been considered when $K_t$ is forbidden as a minor or subdivision.  The problem of maximizing the total number of cliques in the such graph was investigated in a series of papers culminating in~\cite{k5} and~\cite{fox}. Maximizing the number of cliques of a fixed size in a graph without a $K_t$-minor was considered by Wood in~\cite{kcliques} and the problem of maximizing trees in some general graph classes was investigated by Huynh and Wood~\cite{HW}.

Now we introduce some notation which we will require later.
We denote the degree of a vertex~$x$ in a graph $G$ by $d_G(x)$. 
When the graph being considered is obvious, we omit the subscript $G$. 
A path of length~$k$ traversing the vertices $x_1,x_2,\dots,x_{k+1}$ in that order is denoted $x_1x_2\dots x_{k+1}$.
The vertices $x_1$ and $x_{k+1}$ are referred to as terminal vertices, and the remaining vertices are referred to as internal vertices.
Similarly, a cycle of length~$k$ going through the vertices $x_1,x_2,\dots,x_k,x_1$ is denoted by $x_1x_2\dots x_kx_1$.  

A copy of a graph $H$ in a graph $G$ is a  subgraph of $G$ (not necessarily induced), isomorphic to~$H$.  
For graphs $G$ and $H$, we denote by $\N(H,G)$ the number of copies of $H$ in $G$. The neighborhood of a vertex $v$ is denoted by $N(v)$, and the closed neighborhood (that is, $\{v\}\cup N(v)$) is denoted by~$N[v]$. 
Given a graph $G$, the vertex and edge sets of $G$ are denoted by $V(G)$ and $E(G)$, respectively.  
Let~$G$ be a graph and $S \subset V(G)$, then $G[S]$ denotes the induced subgraph of $G$ on the vertex set $S$.  
Let~$C=x_1x_2\dots x_kx_1$ be a cycle in $G$, then $C$ is said to separate the vertices $y,z \in V(G)$ in a planar embedding of $G$ if one of $y$ or $z$ is in the interior of the closed curve formed by embedding of the cycle and the other one is in the exterior. 
For a given graph $G$, if $e = \{v,u\}$ is an edge of $G$, then the contraction of the edge $e$ is the graph obtained from $G$ by replacing the two vertices $\{v,u\}$ with a new vertex $w$ and replacing the edges of the form~$\{v,x\}$ and~$\{y,u\}$ with the edges~$\{w,x\}$ and~$\{y,w\}$ respectively, taking the new edges without multiplicity. 

\section{Proof of the main result}
\subsection{Outline of the proof}
In Subsection~\ref{EL} we prove some basic lemmas necessary for the proof of Theorem~\ref{c5}. We will prove Theorem~\ref{c5} by induction on the number of vertices. 
The proof is split into three parts.
The first part settles the case when the graph contains a vertex $v$ of degree four. 
Next, we settle the case when the graph contains a vertex $v$ of degree three and no vertex of degree four. 
Finally, we settle the case when the minimum degree is five. 

%%%%%%%%%%%%%%%%%%%%%%%
% Note that every maximal planar graph contains a vertex of degree at most five. \textcolor{blue}{Every planar graph contains a vertex of degree at most five, a maximal planar graph has degree minimum degree at least 2.}
%%%%%%%%%%%%%%%%%%%%%%%%

We remove a vertex $v$ from the graph and add some additional edges to make it maximal. 
We then argue that the number of cycles of length-$5$ was not decreased by much.

\subsection{Results and Properties of Planar Graphs}\label{EL}
\input{PGProperties.tex}

\subsection{Proof of Theorem~\ref{c5}}

\begin{proof}[Proof of Theorem \ref{c5}]

Let us define the following function $g(n) = 2n^2-10n+12$ for $n\neq 7$, and $g(7) = 2\cdot7^2-10\cdot7+12+1 = 41$. It is straightforward to verify that the graph $D_n$ contains $g(n)$ copies of $C_5$. Thus we have a lower bound $f(n,C_5) \geq g(n)$, for $n\geq 6$. For $n=5$ the unique maximal planar graph is $D_5$ and it contains six cycles of length~$5$.

For the upper bound, let $G$ be a maximal planar graph that maximizes the number of copies of~$C_5$. 
We may suppose $G$ is a triangulated planar graph on $n$ vertices and that $G$ is embedded in the plane.   
The proof proceeds by induction on $n$ for the statement that $f(n,C_5) \le g(n)$ and equality is attained only for the graphs given in the statement of Theorem~\ref{c5}. 
For the base case $n=5$, we are already done, since $f(5,C_5) = 6 \le 12 = g(5)$. 

For every $(n-1)$-vertex planar graph $G'$ we have 
\begin{equation}\label{Equation:Induction_main}
f(n,C_5) \le f(n-1,C_5) + \N(C_5,G) - \N(C_5,G')\le g(n-1) + \N(C_5,G) - \N(C_5,G'),    
\end{equation}
since $f(n,C_5) =\N(C_5,G)$ and $f(n-1,C_5)  - \N(C_5,G') \geq 0$ .
Note that for $n\not\in \{7,8\}$ we have $g(n)-g(n-1) = 4(n-3)$, while $g(7)-g(6) =4(7-3)+1= 17$ and $g(8)-g(7)=4(8-3)-1=19$. 
Therefore, for every $n$ such that $n\geq 6$ and $n\neq 8$, the desired upper bound of $f(n,C_5)$ follows from Equation~\ref{Equation:Induction_main} hypothesis by defining an $(n-1)$-vertex graph $G'$, such that  
\begin{equation}\label{Equation:Induction_second}
\N(C_5,G) - \N(C_5,G')\leq 4(n-3).    
\end{equation} 
For $n=8$ the desired upper bound of $f(8,C_5)$ follows from Equation~\ref{Equation:Induction_main} hypothesis by defining an $7$-vertex graph $G'$, such that   $\N(C_5,G) - \N(C_5,G')\leq  4(8-3)-1=19$. 

%Since $\delta(G) \in \{3,4,5\}$ we will divide the proof in several cases according to the existence of a vertex of degree $4,3$ or 

%%CT next sentence out of place

Every planar graph contains a vertex of degree at most five. Moreover, every maximal planar graph, with at least four vertices has the minimum degree at least three.
%contains a vertex of degrees three, four, or five. 
The rest of the proof is divided into three major cases. 
Case 1. deals with the case when $G$ contains a vertex $v$ of degree four. Note that that is the case when equality is attained. After we may assume there is no vertex of degree four and the rest of the proof is split into two remaining cases depending on the minimum degree of $G$, which is either three or five. 
Note that Case 2. deals with the case when the minimum degree of $G$ is three and $G$ does not contain any vertex of degree four. 
% In all three cases, we have the property given by Claim~\ref{ncycle} that $G[N(v)]$ has a Hamiltonian cycle.

% In all three cases, we use the following property of a maximal planar graph- for every maximal planar graph $G$, with $n\geq 4$ vertices, $G[N(v)]$ contains a Hamiltonian cycle.   
% \textcolor{blue}{should we expand here?} 

    \textbf{Case 1.} The graph $G$ has a vertex $v$ of degree $4$. 
    
    Let $\{v_1,v_2,v_3,v_4\}$ be the neighborhood of $v$ such that $v_1v_2v_3v_4v_1$ is a cycle.  Note that it is not possible for both of the edges $\{v_1,v_3\}$ and $\{v_2,v_4\}$ to be present in $G$ since otherwise, we would have a copy of $K_5$ in $G$. 

Without loss of generality, suppose that  the edge $\{v_2,v_4\}$ is not present, and consider the graph $G'$ obtained by removing $v$ from $G$ and adding the edge $\{v_2,v_4\}$. 
We are going to bound $\N(C_5,G)-\N(C_5,G')$. Equivalently, we bound the number of copies of $C_5$ in $G$ containing $v$ minus the number of copies of $C_5$ in $G'$ that use the edge $\{v_2,v_4\}$.

We say that a 5-cycle in $G$ containing $v$ is of \emph{type-$k$},  if it contains precisely $k$ vertices which are not in $N[v]$, and we say that a $5$-cycle in $G'$ is of \emph{type-$k$}, if it contains the edge $\{v_2,v_4\}$ and precisely $k$ vertices which are not in $N(v)$.
Let $\N_k$ be the number of cycles of type-$k$ in $G$ minus the number of cycles of type-$k$ in $G'$. Note that $\N_3 \leq 0$ and $\N_k = 0$ for $k\not\in\{0,1,2,3\}$, hence we have  
\[
\N(C_5,G)-\N(C_5,G') = \N_0 + \N_1 + \N_2 + \N_3 \leq \N_0 + \N_1 + \N_2.
\]
Note that $\N_0$ is just the number of $5$-cycles in the graph induced by $N[v]$, thus $\N_0$ is either~$4$ or~$6$ depending on whether $\{v_1,v_3\}$ is present. 
We say that a path is \emph{internal} if all of its vertices are in~$N[v]$. 
We say a path is \emph{external} if all of its non-terminal vertices are not in $N[v]$. 
Note that a path can be neither external nor internal.

We will estimate the number of type-$1$ cycles containing a vertex $x\in V(G)\setminus N[v]$ according to whether the two neighbors of $x$ in the type-$1$ cycles are consecutive or nonconsecutive vertices of the $4$-cycle $v_1v_2v_3v_4v_1$. 

For each $x \in (N(v_i) \cap N(v_{i+1}))\setminus N[v]$ for  $i\in[4]$ (where indices are taken modulo $4$) we have the following cycles of type-$1$ in $G$ containing the path $v_ixv_{i+1}$: $vv_ixv_{i+1}v_{i+2}v$ and $vv_{i-1}v_ixv_{i+1}v$.  
If $\{v_1,v_3\}$ is an edge, then we also have exactly one of the cycles $vv_ixv_{i+1}v_{i-1}v$ or $vv_{i+2}v_ixv_{i+1}v$ (according to whether $i$ is odd or even), moreover in this case, one of the following cycles is a type-$1$ cycle of $G'$: $v_{i+2}v_ixv_{i+1}v_{i-1}v_{i+2}$ or $v_{i-1}v_{i+2}v_ixv_{i+1}v_{i-1}$ (see Figure~\ref{d(v)=4})
%illustrating when $v_i=v_1$ and $v_{i+2}=v_3$). 

Note that whenever $(N(v_i) \cap N(v_{i+2}))\setminus N[v]$ is nonempty we have that $\{v_{i-1},v_{i+1}\}$ is not an edge of $G$.
% \textcolor{blue}{should we expand here?} 
Let $x \in (N(v_i) \cap N(v_{i+2}))\setminus N[v]$.   We have the following four cycles of type-$1$ in $G$ containing $x$: $vv_{i}xv_{i+2}v_{i-1}v$, $vv_{i}xv_{i+2}v_{i+1}v$, $vv_{i-1}v_{i}xv_{i+2}v$, $vv_{i+1}v_{i}xv_{i+2}v$. 
If $i = 1$ (or $i=3$), we have in $G'$ the type-$1$ cycles $v_{2}v_{1}xv_{3}v_{4}v_{2}$ and $v_{4}v_{1}xv_{3}v_{2}v_{4}$ (see Figure~\ref{d(v)=4}). 
For $i=2$ (or $i=4$), we have no type-$1$ cycles in $G'$ using $v_2xv_4$.
\begin{figure}[b]   
\centering
\begin{tikzpicture}

\draw[red,thick] (0,-1) arc(270:90:0.2cm and 1cm);
\draw (1,0) node[right]{$v_3$}  -- (0,0)  -- (0,1) node[above]{$v_2$}  -- (-1,0) node[left]{$v_1$}  -- (0,0) -- (0,-1)node[below]{$v_4$}   -- (1,0) -- (0,1)    (-1,0) -- (0,-1);
\filldraw (1,0) circle (2pt)  (0,0) circle (2pt)  (0,1) circle (2pt)  (-1,0) circle (2pt)  (0,-1) circle (2pt);

\end{tikzpicture}\quad
\begin{tikzpicture}

\draw[blue,thick] (0,2) -- (1,0)  (0,1) -- (-1,0) -- (0,2); 
\draw[blue,thick,dashed](1,0) -- (0,0) -- (0,1);
\draw[red,dotted,thick] (0,-1) arc(270:90:0.2cm and 1cm) (0,-1) -- (1,0) ;
\draw (1,0) node[right]{$v_3$}  (0,0) node[below]{$v$}  (0,1) node[above]{$v_2$}   (-1,0) node[left]{$v_1$}  (0,0)  (0,-1)node[below]{$v_4$}    (1,0) -- (0,1)    (-1,0) -- (0,-1);
\filldraw (1,0) circle (2pt)  (0,0) circle (2pt)  (0,1) circle (2pt)  (-1,0) circle (2pt)  (0,-1) circle (2pt);

\filldraw (0,2) node[left]{$x$} circle (2pt);

\end{tikzpicture}\quad
\begin{tikzpicture}

\draw[blue,thick] (1,-1) -- (1,0) (1,0) arc(0:180:1cm and 2cm)  (0,-1)--(1,-1) ; 
\draw[red,dotted,thick] (0,-1) arc(-90:90:0.2cm and 1cm) (-1,0)--(0,1) ;
\draw[dashed,blue,thick] (-1,0) -- (0,0)--(0,-1);
\draw (1,0) node[right]{$v_3$}  (0,0)node[above]{$v$}  (0,1) node[above]{$v_2$}   (-1,0) node[left]{$v_1$}  (0,0)  (0,-1)node[below]{$v_4$}    (1,0) -- (0,1)    (-1,0) -- (0,-1) (0,-1) -- (1,0) ;
\filldraw (1,0) circle (2pt)  (0,0) circle (2pt)  (0,1) circle (2pt)  (-1,0) circle (2pt)  (0,-1) circle (2pt);

\filldraw (1,-1) circle (2pt) node[right]{$x$};

\end{tikzpicture}\quad
\begin{tikzpicture}

\draw[red,dotted,thick] (0,-1) arc(270:90:0.2cm and 1cm)  (0,-1)--(1,0);
\draw[blue,thick] (.5,1.5) -- (1.5,.5) -- (1,0) (0,1) -- (.5,1.5); 
\draw[blue,thick,dashed] (1,0) -- (0,0) -- (0,1);

\draw (1,0) node[right]{$v_3$}  (0,0)node[below]{$v$}  (0,1) node[above]{$v_2$} --  (-1,0) node[left]{$v_1$}   (0,0)  (0,-1)node[below]{$v_4$}    (1,0) -- (0,1)    (-1,0) -- (0,-1);
\filldraw (1,0) circle (2pt)  (0,0) circle (2pt)  (0,1) circle (2pt)  (-1,0) circle (2pt)  (0,-1) circle (2pt);

\filldraw (1.5,.5) node[above]{$y$} circle (2pt)  (.5,1.5) node[above]{$x$} circle (2pt) ;

\end{tikzpicture}

\caption{The figure shows several cycles in $G$ containing $v$ which can be transformed into cycles in $G'$ with the additional edge $\{v_2,v_4\}$; the new cycles are obtained by replacing the two dashed edges by the two dotted edges.}
\label{d(v)=4}
\end{figure}
Therefore, we have 
\begin{equation}
\label{N1}
\N_1 = \displaystyle \sum_{1\leq i < j \leq 4}h(i,j)\abs{(N(v_i)\cap N(v_j))\setminus N[v]},
\end{equation}
where $h(2,4) = 4$ and $h(i,j) = 2$ for every other pair $(i,j)$ with $1 \le i<j \le 4.$

%where $h(i,j)$ is the number of internal 4-vertex paths from $v_i$ to $v_j$ containing $v$, that is, $h(2,4) = 4$ and $h(i,j) = 2$ for every other pair $(i,j)$, $1 \le i<j \le 4.$

Note that every type-$2$ cycle in $G$ contains two vertices from $\{v_1,v_2,v_3,v_4\}$. 
Therefore there are two kinds of type-$2$ cycles in $G$, those containing two consecutive vertices of the cycle $v_1v_2v_3v_4v_1$ and those containing two opposite vertices of the cycle $v_1v_2v_3v_4v_1$. 
For each type-$2$ cycle in $G$ containing two consecutive vertices of  the cycle $v_1v_2v_3v_4v_1$ we find a corresponding type-$2$ cycle in~$G'$. 
Indeed, let  $vv_ixyv_{i+1}v$ be a  type-$2$  cycle in~$G$, then  either $v_2v_ixyv_{i+1}v_2$ or $v_4v_ixyv_{i+1}v_4$ is a type-$2$ cycle in~$G'$ depending on the parity of $i$, see Figure~\ref{d(v)=4}.  
Note that each of these type-$2$ cycles in~$G'$ contains exactly three vertices from  $\{v_1,v_2,v_3,v_4\}$, and they are consecutive on the cycle.
Hence, for estimating $\N_2$ we need only to consider type-$2$ cycles in $G$ containing two opposite vertices of the cycle $v_1v_2v_3v_4v_1$ and type-$2$ cycles of $G'$ not containing exactly three vertices from~$\{v_1,v_2,v_3,v_4\}$.

\textbf{Case 1.1} The edge $\{v_1,v_3\}$ is present. 

If $n=6$, then $G$ is isomorphic to $E_6$, which contains $20$ cycles of length~$5$. 
Hence we may suppose $n\geq 7$.

In this case there is no path joining $v_2$ and $v_4$ without using $v$, $v_1$ or $v_3$. 
It follows that the number of type-$2$ cycles in $G$ is at most the number of  external paths of length~$3$ from $v_1$ to $v_3$.  
Note that there is at most one vertex in $V(G) \setminus N[v]$ which is adjacent to $v_1$, $v_2$ and $v_3$ as well as at most one vertex adjacent to $v_3$, $v_4$ and $v_1$. 
Every other vertex is adjacent to at most two vertices from $N(v)$. 

We will assume that there exists a vertex $u\neq v$, adjacent to $v_1,v_2,v_3$ and a vertex  $w \not\in \{v, u\}$, adjacent to $v_3,v_4,v_1$, since the other cases follow from a similar argument. 
% \textcolor{blue}{should we expand here?} 
We have that $\N_0 = 6$ and $\N_1 \leq 12 + 2(n-7)$ by~\eqref{N1}. Indeed, this is true since both $u$ and $w$ are each contained in six type-$1$ cycles, and every other vertex of $V(G) \setminus N[v]$ is contained in at most two  type-$1$ cycles. Equality holds if and only if every vertex from $V(G)\setminus N[v]$ other than $u$ and $w$ is adjacent to precisely two vertices of $N(v)$.  

By applying Lemma~\ref{2(k-1)} to the  regions determined by the triangles $v_1v_2v_3v_1$ and $v_3v_4v_1v_3$ (containing $u$ and $w$, respectively), we have at most $2(n-7)$ external paths of length~$3$ from $v_1$ to $v_3$.  There are at least two type-$2$ cycles in $G'$ containing exactly two vertices of  $\{v_1,v_2,v_3,v_4\}$, namely $uv_2v_4wv_3u$ and $uv_2v_4wv_1u$, see Figure~\ref{d(v)=4b}. Hence $\N_2 \leq 2(n-7) -2$. Thus, we have 
\[
\N(C_5,G) - \N(C_5,G') = \N_0 + \N_1 + \N_2 \leq 4(n-3),\]
and we are done by Equation~\ref{Equation:Induction_second}.
Note that if every vertex in $V(G)\setminus N[v]$ besides $u$ and $w$ is adjacent to exactly two vertices of $N(v)$, then there are $3n-6-9-2(n-5)-2= n-7$ edges not incident with any vertex in $N(v)$. Since every $4$-vertex external path from $v_1$ to $v_3$ is of the form $v_1xyv_3$, where the edge $\{x,y\}$ is not incident to a vertex from $N[v]$, we have a bound of $2(n-7)$ on the number of such paths, and equality holds only if every such vertex in $V(G)\setminus N[v]$ is adjacent to both $v_1$ and $v_3$.
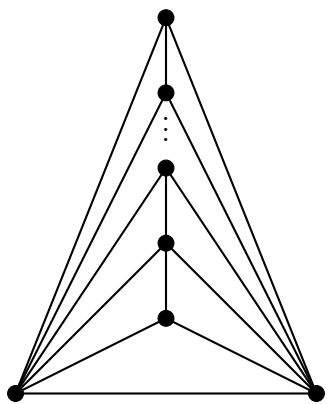
Therefore, either $G$ is the graph $E_n$, which has $2n^2-10n+8<g(n)$ copies of $C_5$, or $\N(C_5,G) - \N(C_5,G') < 4(n-3)$, and we are done by Equation~\ref{Equation:Induction_second}. 
Hence for $n\neq 8$ we are done. On the other hand, for $n=8$ it is simple to observe that  $G'$ is not isomorphic to $D_7$, so $\N(C_5,G')< 40$ since the induction hypothesis that asserts that the extremal graph is unique for $n=7$, and we are done.

\textbf{Case 1.2} Suppose the edge $\{v_1,v_3\}$ is not present and   $N(v_1)\cap N(v_2) \cap N(v_3) \cap N(v_4)=\{v\}$.  

In this case $\N_0 = 4$.
Since $G$ contains no $K_5$-minor there is no external path of length two from $v_{i}$ to $v_{i+2}$ for some  $i\in\{1,2\}$. 
Thus, without loss of generality, we may assume $N(v_2)\cap N(v_4) = \{v\}$.  
Since $G$ is $K_{3,3}$-free there is at most one vertex $u\not =v$ adjacent to $v_1,v_2,v_3$, and there is at most one vertex $w\not=v$ adjacent to $v_3,v_4,v_1$.  
We will assume that indeed there exist such vertices $u$ and $w$ and obtain the desired upper bound  $\N(C_5,G) - \N(C_5,G')\leq 4(n-3)$ and we are done by Equation~\ref{Equation:Induction_second}. we omit the cases when one of these vertices does not exist since they follow from the same argument.
Since if either the vertex $u$ or the vertex $w$ does not exist the number of $C_5$'s incident with $v$ will decrease.  
Indeed, if there is no vertex $u$ incident with $v_1,v_2,v_3$ then there is at most one edge $\{u',u''\}$ such that $u'$ is adjacent with $v_1$, $v_2$ and $u''$ is adjacent with $v_2$, $v_3$. After contracting the edge $\{u',u''\}$ to a vertex the number of five cycles incident with $v$ decreases by one, the number of vertices of the graph decreases by one and the resulting graph has a vertex incident with all three vertices $v_1$, $v_2$ and $v_3$. Thus the result assuming the existence of $u$ implies that case.   If edge $\{u',u''\}$ such that $u'$ is adjacent with $v_1$, $v_2$ and $u''$ is adjacent with $v_2$, $v_3$ is not an edge of $G$ then no argument in the following proof changes. 

Let $X$ be the set of vertices in $V(G)\setminus N(v)$ which are adjacent to precisely two vertices of~$N(v)$, and let $Y$ be the set of vertices that are adjacent to just one vertex of $N(v)$.  
% Both $u$ and $w$ contribute to six type-$1$ cycles of $\N_1$ and every vertex of $X$ contributes to two. 
% \textcolor{blue}{The previous line seems unnecessary, if we invoce (3).}
By~\eqref{N1} we have that 
$\N_1 \leq 12 + 2\abs{X}.$

 \textbf{Case~1.2.1} The vertices $w$ and $u$ are adjacent. 
 
 We have an external path of length~$3$ from $v_2$ to $v_4$, namely $v_2uwv_4$.
 Furthermore, we have two external paths of length~$3$ from $v_1$ to $v_3$, namely $v_1uwv_3$ and $v_1wuv_3$. 
For each vertex $z \in V(G)\setminus N[v]\setminus\{u,w\}$, 
there is no external path of length~$3$ from $v_2$ to $v_4$ incident with $z$. 
Any  external path of length~$3$ from $v_1$ to $v_3$ incident with $z$ contains either $u$ or $w$. 
Moreover, if $z$ is incident with two such paths, then $z$ must be adjacent to both $u$ and $w$.  
If $z$ is adjacent to both $u$ and $w$, then there is one type-$3$ cycle in $G'$ given by $zuv_2v_4wz$. 
Note that it contains only two vertices from $N(v)$, see Figure~\ref{d(v)=4b}. 
Therefore $z$ contributes at most one to $\N_2+\N_3$.
For every $x\in X$, since $x$ must be adjacent to two consecutive vertices of $N(v)$, we have a type-$2$ cycle in $G'$.
For instance, if $x$ is adjacent to $v_1$ and $v_2$, then we have the $5$-cycle $xv_2v_4wv_1x$, and this cycle contains three vertices of $N(v)$ which are not consecutive on the cycle. 
Hence the vertices from $X$ contribute nothing to $\N_2+\N_3$. %CT - formalize
Since the only  external paths of length~$3$ incident with either $u$ or $w$ from $v_i$ to $v_{i+2}$ for $i\in\{1,2\}$ are $v_1uwv_3$, $v_2uwv_4$ and $v_3uwv_1$,  
% \textcolor{blue}{confusing+}
on the other hand, we have the following type-$2$ cycles in $G'$: $uv_2v_4wv_3u$ and $uv_2v_4wv_1u$, which contain nonconsecutive vertices of  $N(v)$  (see Figure~\ref{d(v)=4b}), 
then the vertices $u$ and $w$ together contribute one to the sum  $\N_2+\N_3$. 
Finally, we have 
\[
\N_2+\N_3 \leq 1 + (n-7) - \abs{X}.
\]
Therefore if $n\geq 8$, we have 
\[
\N(C_5,G) - \N(C_5,G') \leq 17 + (n-7)+\abs{X}\leq 17+2(n-7)<4(n-3),
\]
and we are done by Equation~\ref{Equation:Induction_second}.
If $n = 7$,  then $G$ is isomorphic to $D_7$ and so $\N(C_5,G)=41$.   
% \textcolor{blue}{should we expand here?}
For $n = 8$, we have that $\N(C_5,G) - \N(C_5,G') \leq 19$. However, it is easy to observe that $G'$ is not isomorphic to $D_7$. Therefore $\N(C_5,G)\leq \N(C_5,D_7)-1+19=59<60$.

\begin{figure}[thb]   
\centering

\begin{tikzpicture}
\draw[blue,dashed] (1,0) -- (0,0) -- (-1,0);
\draw[red,thick] (0,2) -- (-2,0) --(0,-2) (0,1) --(0,2) (0,-1) --(0,-2)  ;
%\draw[red,thick] (0,-2) -- (-1,0) -- (-.33,1) -- (0,1);
\draw[blue,thick] (0,2) arc (90:-90:1.6cm and 2cm)    (-1,0) -- (0,2)  (0,-2) -- (1,0) ; 
% \draw[blue,thick,dashed](1,0) -- (0,0) -- (0,1) ;
\draw[red,dotted,thick] (0,-1) arc(270:90:0.2cm and 1cm) ;
\draw (1,0) node[right]{$v_3$}  (0,0) node[below]{$v$}  (0,1) node[right]{$v_2$}   (-1,0) node[left]{$v_1$}  (0,0)  (0,-1)node[right]{$v_4$}    (1,0) -- (0,1)    (-1,0) -- (0,-1);
\filldraw (1,0) circle (2pt)  (0,0) circle (2pt)  (0,1) circle (2pt)  (-1,0) circle (2pt)  (0,-1) circle (2pt);
\filldraw (0,2) node[left]{$u$} circle (2pt)  (-2,0) circle (2pt)  node[left]{$z$};
\filldraw (0,-2) node[left]{$w$} circle (2pt);
\draw  (0,-1) -- (1,0)  (0,1) -- (-1,0) (0,2) -- (1,0)  (0,-2) -- (-1,0); 
% \filldraw (-.33,1) circle (2pt) ;
% \node at (-.25,1.2){$x$};
\end{tikzpicture}\quad
\begin{tikzpicture}
%\draw[blue,dashed] (1,0) -- (0,0) -- (-1,0);
\draw[red,thick] (0,2) -- (1,0) --(0,-2) (0,1) --(0,2) (0,-1) --(0,-2)  ;
%\draw[red,thick] (0,-2) -- (-1,0) -- (-.33,1) -- (0,1);

% \draw[blue,thick,dashed](1,0) -- (0,0) -- (0,1) ;
\draw[red,dotted,thick] (0,-1) arc(270:90:0.2cm and 1cm) ;
\draw (1,0) node[right]{$v_3$}  (0,0) node[below]{$v$}  (0,1) node[right]{$v_2$}   (-1,0) node[left]{$v_1$}  (0,0)  (0,-1)node[right]{$v_4$}    (1,0) -- (0,1)    (-1,0) -- (0,-1);
\filldraw (1,0) circle (2pt)  (0,0) circle (2pt)  (0,1) circle (2pt)  (-1,0) circle (2pt)  (0,-1) circle (2pt);
\filldraw (0,2) node[left]{$u$} circle (2pt) ;
\filldraw (0,-2) node[left]{$w$} circle (2pt);
\draw  (0,-1) -- (1,0)  (0,1) -- (-1,0)   (0,-2) -- (-1,0) (-1,0) -- (0,2) ; 
% \filldraw (-.33,1) circle (2pt) ;
% \node at (-.25,1.2){$x$};
\end{tikzpicture}\quad

\caption{The first picture shows a type-$3$ cycle in $G'$ in Case~1.1 when $u,w$ 
% \textcolor{blue}{$w$ instead of $v$?}
have a common neighbor and a type-$2$ cycle in $G$ when the edge $\{u,w\}$ is present. The second picture shows one of the two type-$2$ cycles in $G'$ that uses both $u$ and $w$.} 
% The last two pictures show a type-$2$ cycle in Case~1.c. There is one such cycle for every common neighbor of $u$ and $v_i$ not in $N(v)$. The last picture shows a type-$2$ cycle from Case~1.c which occurs in $G'$ for each $x\in X$, without using the edge $\{u,x\}$.}
\label{d(v)=4b}
\end{figure}
 \textbf{Case~1.2.2}
The vertices $u$ and $w$ are not adjacent. 

For this case to be possible  we have $n\geq 8$.
Furthermore, if $n=8$ the graph $G$ is well defined and it is isomorphic to $D_8$. 
Hence, from here on we may assume $n>8$. 

Note that there is no external path of length~$3$ from $v_2$ to $v_4$. 
Now we show that there are at most $2(n-6)$ external paths of length~$3$ from $v_1$ to $v_3$.  
Consider an auxiliary graph $G''$ obtained from $G$ by removing $v$ and adding an edge between the vertices $v_1$ and $v_3$.
Note that all external paths of length~$3$ from $v_1$ to $v_3$ which exist in $G$ also exist in $G''$. By Lemma~\ref{2(k-1)}, there are $2(n-4)$ paths of length~$3$ in $G''$ from $v_1$ to $v_3$. 
However, some of those paths are not external in $G$.
In particular, the following are non-external paths of length three in $G$ from $v_1$ to $v_3$: $v_1v_2uv_3$, $v_1v_4wv_3$, $v_1uv_2v_3$ and $v_1wv_4v_3$. 
Hence there are at most $2(n-6)$ external paths of length~$3$ from $v_1$ to $v_3$.
  
There are at least two additional type-$2$ cycles in $G'$, namely $uv_2v_4wv_3u$ and $uv_2v_4wv_1u$ (they contain three vertices of $N(v)$ and they are not consecutive). 
Hence we have  $\N_2 \leq 2(n-6)-2$. 
Therefore we have the desired upper bound
\[
\N(C_5,G) - \N(C_5,G') \leq 14 +2(n-6) + 2\abs{X} \leq  4(n-3),
\]
and we are done by Equation~\ref{Equation:Induction_second}.
Moreover, equality holds only if $\abs{X} = n-7$. 
That is if every vertex in $V(G)\setminus N[v]$ except for $u$ and $w$ is adjacent to precisely two vertices of $N(v)$. 
Thus, there are $n-6$ edges not incident with a vertex from $N(v)$, so the number of external paths of length~$3$ from $v_1$ to $v_3$ is at most $2(n-6)$.
Equality holds only if each vertex in $V(G) \setminus N(v)$ is adjacent to both $v_1$ and $v_3$. Consequently, $G$ is isomorphic to $D_n$.

\textbf{Case 1.3} We have $N(v_1)\cap N(v_2)\cap N(v_3)\cap N(v_4)=\{v,u\}$ for some vertex $u\neq v$. 

In this case the edge $\{v_1,v_3\}$ is not present.
If $n=6$ then $G$ is isomorphic to $D_6$ and we are done. 
Suppose $n\geq 7$. 
Note that there is no vertex distinct from $u$ in $V(G) \setminus N[v]$ adjacent to $v_i$ and $v_{i+1}$ for some $i\in\{1,2\}$.
Let $X$ be the set of vertices that are adjacent to precisely two vertices of $N(v)$, and let $Y$ be the set of vertices that are adjacent to exactly one vertex of $N(v)$.
% Then we have that 
Since 
\[\displaystyle \sum_{1\leq i < j \leq 4}h(i,j)\abs{(N(v_i)\cap N(v_j))\setminus N[v]} =|X|+4,\] and $N(v_2)\cap N(v_4)=\{u,v,v_1,v_3\}$, it follows by~\eqref{N1} that 
\begin{align*}
\N_1 &= \displaystyle \sum_{1\leq i < j \leq 4}h(i,j)\abs{(N(v_i)\cap N(v_j))\setminus N[v]} \\&= 2\sum_{1\leq i < j \leq 4}\abs{(N(v_i)\cap N(v_j))\setminus N[v]} + 2|N(v_2)\cap N(v_3)\setminus N[v]| \\&=
14 + 2\abs{X}.
    \end{align*}

% \[
% \N_1 = 14 + 2\abs{X}.  
%  \text{\textcolor{blue}{we are using $N(v_2)\cap N(v_4)=\{u,v\}$}}

% \]

Observe that every vertex $y\in Y$ is in at most one external path of length three between opposite vertices of the cycle $v_1v_2v_3v_4v_1$. 
Moreover, such an external path of length three contains the vertex~$u$, see Figure~\ref{d(v)=4c}.
%\textcolor{blue}{should we expand here?}
Similarly, every vertex $x$ from $x\in X$ can be contained in at most two external paths of length three between opposite vertices of the cycle $v_1v_2v_3v_4v_1$. 
Moreover, if such a path exists, then $x$ is adjacent to~$u$.
The vertex $u$ can be adjacent to at most one vertex of $X$ in the region bounded by $uv_{i}v_{i+1}u$.
Thus there are at most $2 \min\{4,\abs{X}\} \leq 8$ external paths of length three between opposite vertices of the cycle $v_1v_2v_3v_4v_1$ containing a vertex from $X$.  
Also note that every $x \in X$ is in a type-$2$ cycle of $G'$, containing three vertices of $N(v)$ not in consecutive order. 
In particular, if $x$ is adjacent to $v_{2i}$ and $v_{2i+1}$, then $xv_{2i}v_{2i+2}uv_{2i+1}x$ is a five-cycle in $G'$, see Figure~\ref{d(v)=4c}. 
Therefore $\N_2 \leq \abs{Y} - \abs{X} + 8$, and it follows that \[\N(C_5,G)-\N(C_5,G') \leq 26 + \abs{X} + \abs{Y}\leq 20 +n. \]
If $n\geq 11$, then $\N(C_5,G)-\N(C_5,G')<4(n-3)$  and we are done by Equation~\ref{Equation:Induction_second}. From here we suppose $6\leq n\leq 10$.

If there is precisely one vertex $x$ inside the region $v_iv_{i+1}u$ for some $i\in[4]$, then this vertex $x$ is adjacent to the three vertices $v_i, v_{i+1}, u$. 
Moreover $N(v_i)\cap N(v_{i+1}) \cap N(u)=\{x\}$.  
Consider an auxiliary graph $G''$, which is obtained from $G$ by removing $x$. 
By the second part of Lemma~\ref{triangle} for the graph $G''$ and the face $v_iv_{i+1}u v_i$, 
there are at most $4(n-3)-1$  paths of length~$3$ with both terminal vertices in $v_iv_{i+1}u$. 
Hence we are done since the number of $5$-cycles incident to $v$ in $G$ is equal to the number of paths of length~$3$ with both terminal vertices in $v_iv_{i+1}u$.

If there are precisely two vertices $x,y$ in the region $uv_iv_{i+1}u$, for some $i\in[4]$, then one of these two vertices, say $x$, will have degree~$4$ and it will have the neighborhood $\{v_i,v_{i+1},y,u\}$. 
Note that $G[\{v_i,v_{i+1},y,u\}]$ contains a triangle $v_iv_{i+1}uv_i$.
Since $x$ is a vertex of degree $4$ such that $N(x)$ contains a four-cycle with a diagonal, it satisfies the condition of Case~1.1, and we are done.  

If $n\in \{7,8\}$, then there will be a region $uv_{i}v_{i+1}u$ containing one or two vertices, for some $i\in[4]$, hence we are done. 

If $n\in\{9,10\}$, then all vertices from $V(G)\setminus N[v]\setminus \{u\}$ are in the same region $v_i v_{i+1} u$, for some $i\in[4]$. 
Thus the number of the external paths of length three between opposite vertices of the cycle $v_1v_2v_3v_4v_1$ incident with a vertex of $X$ is $2$ instead of $2 \min\{4,\abs{X}\} \leq 8$.  
Hence we have  $\N_2 \leq \abs{Y} - \abs{X} + 2$.  
It follows that 
\[\N(C_5,G)-\N(C_5,G') \leq 20 + \abs{X} + \abs{Y}\leq 14 +n < 4(n-3),\]
and we are done by Equation~\ref{Equation:Induction_second}.

\begin{figure}[thb]   
\centering

\begin{tikzpicture}
\draw[blue,thick]  (-1,0) -- (0,2)  (1,-1) -- (1,0)   (1,-1) arc(-61.58:118.42:1.1cm and 1.69cm); 
\draw[blue,dashed] (1,0) -- (0,0) -- (-1,0) ;
\draw  (0,1) --(0,2)  (0,2) -- (1,0);
%\draw[red,thick] (0,-2) -- (-1,0) -- (-.33,1) -- (0,1);
% \draw[blue,thick,dashed](1,0) -- (0,0) -- (0,1) ;
%\draw[red,dotted,thick] (0,-1) arc(270:90:0.2cm and 1cm) ;
\draw (1,0) node[right]{$v_3$}  (0,0) node[below]{$v$}  (0,1) node[right]{$v_2$}   (-1,0) node[left]{$v_1$}  (0,0)  (0,-1)node[below]{$v_4$}    (1,0) -- (0,1)    (-1,0) -- (0,-1) (1,-1) node[left]{$y$};
\filldraw (1,0) circle (2pt)  (0,0) circle (2pt)  (0,1) circle (2pt)  (-1,0) circle (2pt)  (0,-1) circle (2pt);
\filldraw (0,2) node[above]{$u$} circle (2pt) ;
\filldraw (0,-1.9) ;
\draw  (0,-1) -- (1,0)  (0,1) -- (-1,0)  (0,2) arc (90:270:1.8cm and 1.5cm);   
% \filldraw (-.33,1) circle (2pt) ;
% \node at (-.25,1.2){$x$};
\filldraw (1,-1) circle (2pt);
\end{tikzpicture}\quad
\begin{tikzpicture}

\draw[red,thick]  (0,1) --(0,2) (0,-1)--(1,-1) (0,2) -- (1,0) (1,-1) -- (1,0);
%\draw[red,thick] (0,-2) -- (-1,0) -- (-.33,1) -- (0,1);
% \draw[blue,thick,dashed](1,0) -- (0,0) -- (0,1) ;
\draw[red,dotted,thick] (0,-1) arc(270:90:0.2cm and 1cm) ;
\draw (1,0) node[right]{$v_3$}  (0,0) node[below]{$v$}  (0,1) node[right]{$v_2$}   (-1,0) node[left]{$v_1$}  (0,0)  (0,-1)node[below]{$v_4$}    (1,0) -- (0,1)    (-1,0) -- (0,-1);
\filldraw (1,0) circle (2pt)  (0,0) circle (2pt)  (0,1) circle (2pt)  (-1,0) circle (2pt)  (0,-1) circle (2pt);
\filldraw (0,2) node[above]{$u$} circle (2pt) ;
\filldraw (0,-1.9) ;
\draw  (0,-1) -- (1,0)  (0,1) -- (-1,0)  (0,2) arc (90:270:1.8cm and 1.5cm);   
% \filldraw (-.33,1) circle (2pt) ;
% \node at (-.25,1.2){$x$};
\filldraw (1,-1) circle (2pt) node[right]{$x$};
\draw  (-1,0) -- (0,2) ; 
\end{tikzpicture}

\caption{The first picture shows a type-$3$ cycle in $G'$ in Case~1.1 when $u,w$ 
% \textcolor{blue}{$w$ instead of $v$?}
% have a common neighbor and a type-$2$ cycle in $G$ when the edge $\{u,w\}$ is present. The second picture shows one of the two type-$2$ cycles in $G'$ that uses both $u$ and $w$. 
The last two pictures show a type-$2$ cycle in Case~1.3. There is one such cycle for every common neighbor of $u$ and $v_i$ not in $N(v)$. The last picture shows a type-$2$ cycle from Case~1.3 which occurs in $G'$ for each $x\in X$, without using the edge $\{u,x\}$.}
\label{d(v)=4c}
\end{figure}

\textbf{Case 2.} The graph $G$ has no vertex of degree $4$ but $G$ has a vertex $v$ of degree $3$. 

Let $\{v_1,v_2,v_3\}$ be the neighborhood of $v$. 
Then $v_1v_2v_3v_1$ is a triangle, and we will assume $v$ is in the interior of the triangle of the plane graph $G$. 
Let $G'$ be the graph induced by $G[V(G) \setminus \{v\}]$.
The number of $5$-cycles in $G$ containing $v$ is precisely the number of paths of length~$3$ with end vertices in~$N(v)$. 
Note that, the vertices of $N(v)$ are the vertices of a triangular face in $G'$, so we may apply Lemma~\ref{triangle} for the face $G'[N(v)]$ in $G'$.
We now consider two subcases depending on whether or not there exists a vertex distinct from $v$ adjacent to all three vertices of $N(v)$ in $G$.

\textbf{Case 2.1} There is no vertex $u \neq v$ adjacent to the vertices $v_1$, $v_2$ and $v_3$. 

By Lemma~\ref{triangle} there are at most $4(n-3)-1$ cycles of length five containing $v$. 
Thus for $n\neq 8$ we have  $\N(C_5,G)<2n^2-10n+12$ by induction. 
For $n=8$ we have $\N(C_5,G) \leq 2n^2-10n+12$.
Equality may hold only if $G'$ is isomorphic to $D_7$. 
Hence if equality holds, then $G$ is obtained from $D_7$ by adding a vertex in a triangular face; it is straightforward to check that all such graphs contain at most $57<g(8)$ cycles of length~$5$.

\textbf{Case 2.2} There is a vertex $u \neq v$, which is adjacent to the three vertices of~$N(v)$.

Since there is no vertex of degree~$4$, the degrees of the vertices $v_1$, $v_2$ and $v_3$ are at least five. 
Hence at least two of the following three bounded regions of $G'$ contain a vertex:  $uv_1v_2u$, $uv_2v_3u$, $uv_3v_1u$. 
Note that for $n=6$ there is no such graph. 
For $n=7$ there is the unique planar graph with these properties, namely the graph on the left in Figure~\ref{exc0}.
It contains~$36$ cycles of length~$5$, and we are done. 
For $n=8$ there is also a unique planar graph with these properties, namely the graph on the right in Figure~\ref{exc0}. 
This graph is isomorphic to $A_8$, and it contains $60$ cycles of length~$5$. 
From here on we assume that $n\geq 9$. 

\begin{figure}[b]

\centering
\begin{tikzpicture}[scale=.5 ] %7
\filldraw (0,0) circle (5pt) node[left]{$v$} -- (30:1cm)  circle (5pt) -- (150:1cm) circle (5pt)  -- (270:1cm) circle (5pt)  (0,0) -- (150:1cm) (270:1cm)--(30:1cm) (270:1cm) -- (0,0) ;
\filldraw   (270:1cm) -- (210:2cm) circle (5pt)  --  (150:1cm);
\filldraw   (30:1cm) -- (330:2cm) circle (5pt)  --  (270:1cm);
\draw (270:1cm) arc(230:490: 2.2cm and 2.6cm);
\draw[shift= {(330:2cm)},rotate=-68] (0,0) arc (0:180: 2.18cm and 1.2cm);
\draw[shift= {(210:2cm)},rotate=-112] (0,0) arc (0:-180: 2.18cm and 1.2cm);
\filldraw   (30:1cm) -- (90:3cm) node[above]{$u$} circle (5pt)  --  (150:1cm);
\end{tikzpicture}\qquad
\begin{tikzpicture}[scale=.5 ] %8
\filldraw (0,0) node[left]{$v$} circle (5pt)  -- (30:1cm)  circle (5pt) -- (150:1cm) circle (5pt)  -- (270:1cm) circle (5pt)  (0,0) -- (150:1cm) (270:1cm)--(30:1cm) (270:1cm) -- (0,0);
\filldraw   (270:1cm) -- (210:2cm) circle (5pt)  --  (150:1cm);
\filldraw   (30:1cm) -- (330:2cm) circle (5pt)  --  (270:1cm);
\filldraw    (30:1cm) --  (90:2cm) circle (5pt) -- (150:1cm) ;
\draw (270:1cm) arc(230:490: 2.2cm and 2.6cm);

\draw[shift= {(330:2cm)},rotate=-68] (0,0) arc (0:180: 2.18cm and 1.2cm);
\draw[shift= {(210:2cm)},rotate=-112] (0,0) arc (0:-180: 2.18cm and 1.2cm);

\filldraw   (30:1cm) -- (90:3cm) circle (5pt)  --  (150:1cm);
\filldraw   (90:2cm)  --(90:3cm) node[above]{$u$};

\end{tikzpicture}

\caption{The unique planar graphs with $7$ and $8$ vertices which have minimum degree three and no vertex of degree four.} 
\label{exc0}
\end{figure}

Since $G$ contains no vertex of degree~$4$, it follows that $G'$ contains at most $3$ vertices of degree four.
Hence $G'$ is not isomorphic to $D_{n-1}$. 
Since each face of the graphs $A_8$ and  $A_{11}$ contains a vertex of degree three and $G$ contains no vertex of degree four, it follows that $G'$ is not isomorphic to neither the graph $A_8$ nor $A_{11}$.
Thus we have $\N(C_5,G') \leq (n-1)^2 - 10(n-1) +11$, and  to finish the proof of this case it is enough to show that $v$ is in at most $4(n-3)$ cycles of length~$5$, and we are done by Equation~\ref{Equation:Induction_second}.

If there is at least one vertex that is not adjacent to any vertex in $N(v)$, then such a vertex does not appear in a $4$-vertex path with terminal vertices in $N(v)$. 
Thus, by Lemma~\ref{triangle}, $v$ is in at most $4(n-3)$  cycles of length~$5$ and we are done by Equation~\ref{Equation:Induction_second}. Thus, we may assume that every vertex is adjacent to at least one vertex of $N(v)$.
Let \[X= \{x\in V(G) \setminus N(v): \abs{N(x)\cap N(v)} = 2\}\] and \[Y = \{y\in V(G):\abs{N(y)\cap N(v)} = 1\}.\]

Note that every vertex of $Y$ is in one of the bounded triangular regions of $G'$: $uv_1v_2u$, $uv_2v_3u$, $uv_3v_1u$. Furthermore, it is adjacent to at most two vertices of $X \cup \{u\}$, since $G$ is  $K_{3,3}$-free. 
Even more, if $y \in Y$ and  $N(y) \cap N(v) = \{v_i\}$, then we may suppose that every neighbor of $y$ is also a neighbor of $v_i$.
Otherwise, if there is a  vertex $y'$ adjacent to $y$ and not adjacent to $v_i$, then the 
graph $G''$ obtained from $G'$ by contracting the edge $\{y,w\}$ contains one more path of length three with terminal vertices in $N(v)$ than $G'$ does (this was shown in the proof of Lemma~\ref{triangle}).
By  Lemma~\ref{triangle} the number of paths of length three in $G''$ with terminal vertices in $N(v)$ is at most $4(n-3)$. 
Hence the number of paths of length three in $G'$ with terminal vertices in $N(v)$ is at most $4(n-3)-1$, and we are done by Equation~\ref{Equation:Induction_second}.
Thus, we will assume for every $y \in Y$ such that  $N(y) \cap N(v) = \{v_i\}$, every neighbour of $y$ is also adjacent to the vertex $v_i$. 
Thus every path of length three in $G'$ with terminal vertices in $N(v)$ and containing $y$ also contains a vertex from $X$. 
Hence each vertex of $Y$ is in at most three paths of length~$3$ with terminal vertices in $N(v)$. Even more, this bound is sharp only if $y$ is adjacent to $u$ and a vertex of $X$. Otherwise, $y$ is in at most two  paths of length~$3$ with terminal vertices in $N(v)$. 

Let us apply Lemma~\ref{triangle} for the graph  $G[V(G')\setminus Y]$ and face $v_1v_2v_3v_1$.
This graph contains at most  $4(n-1-\abs{Y})$ paths of length three with terminal vertices on the face. 
Hence there are at most $4(n-2-\abs{Y})+3\abs{Y}=4(n-3)-(\abs{Y}-4)$ paths of length~$3$ with terminal vertices in $N(v)$.
Thus if 
$\abs{Y}\geq 4$, we are done by induction. 
We  distinguish two further subcases based on whether or not at least one of the regions $uv_iv_{i+1}u$,  $i\in \{1,2,3\}$ contains no vertex.

\textbf{Case 2.2.1}
One of these regions contains no vertex and assume without loss of generality that this region is $uv_1v_2u$. 

By Lemma~\ref{triangle} the number of paths of length three with terminal vertices on the face $v_1v_2v_3v_1$ in the graph obtained from $G[V(G')\setminus Y]$ by adding a vertex $w$ of degree three inside the face $uv_1v_2u$ is at most $4(n-1-\abs{Y})$. 
Note that the number of paths of length three containing $w$ with terminal vertices on the face $v_1v_2v_3v_1$ is six.
Hence we obtain a new improved upper bound $4(n-1-\abs{Y})+3\abs{Y}-6=4(n-3)+2-\abs{Y}$ on the number of paths of length three with terminal vertices on the face $v_1v_2v_3v_1$ in $G'$.
Hence we may assume that $\abs{Y}\leq 1$, or we are done by induction.
Moreover, if $\abs{Y} = 1$, then the vertex of $Y$ must be adjacent to $u$, since otherwise, the number of paths of length three with terminal vertices in $N(v)$ is at most two.  
The number of paths of length three with terminal vertices on the face $v_1v_2v_3v_1$ in $G'$ is at most $4(n-3)$, and we are done by Equation~\ref{Equation:Induction_second}.

Let for each   $i,j \in \{1, 2, 3\}$ such that $i<j$ let $X_{i,j}$ be the set of vertices from $X$ incident with $v_i$ and $v_j$. 
Then $X_{i,j}$ together with $u$ induces a path forest by Lemma~\ref{pforest}.
Since $G'$ is a triangulation, $\abs{Y}\leq 1$ and $V(G')=X\cup Y\cup \{v_1,v_2,v_3\}$ it follows that $X_{i,j}$ together with $u$ induces a path.
Since the region $uv_1v_2u$ is empty, $G'[X\cup\{u\}]$ is a path. 
Each non-terminal vertex of the path in $G'[X\cup\{u\}]$ distinct from $U$ has degree $4$ in $G'[V(G)\setminus Y]$. 
The number of such vertices is at least $n-8$.
Recall $G'$ has no vertex of degree $4$, a vertex from $Y$ is adjacent to at most one vertex from $X$ and $\abs{Y}\leq 1$, hence $n-8\leq 1$. 
Thus if $n\geq 10$ then we are done. 

For $n = 9$ we have  $Y \neq \emptyset$, since otherwise the path $G'[X\cup\{u\}]$ contains four vertices and $G'$ contains a vertex of degree $4$, a contradiction. 
There exist two different planar graphs with such properties, see the two graphs on the left in Figure~\ref{exc}. 
It can be checked that those graphs contain $79$ and $80$ cycles of length~$5$, respectively.
Hence we are done since $g(9)=84$.

\begin{figure}[t]
\centering
\begin{tikzpicture}[scale=.5 ]%9 a)
\filldraw (0,-.2) circle (5pt) node[right]{$x_2$} -- (30:1cm)  circle (5pt)  -- (150:1cm) circle (5pt)  -- (270:1cm) circle (5pt)  (0,-.2)  -- (150:1cm) (270:1cm)node[below]{$x_3$}--(30:1cm) (270:1cm) -- (0,-.2);
\filldraw   (270:1cm) -- (210:2cm) node[left]{$y$} circle (5pt)  --  (150:1cm);
%\filldraw   (30:1cm) -- (330:2cm) circle (5pt)  --  (270:1cm);
\filldraw   (90:5cm) circle (5pt)node[above]{$u$};
\draw[shift= {(30:1cm)},rotate=-80] (0,-.2) arc (0:180: 2.3cm and 1cm);
\draw[shift= {(150:1cm)},rotate=-100] (0,-.2) arc (0:-180: 2.3cm and 1cm);
\draw (270:1cm) arc(240:480: 2.2cm and 3.5cm);
%\draw[shift= {(330:2cm)},rotate=-74] (0,0) arc (0:180: 3.13cm and 1.2cm);
\draw[shift= {(210:2cm)},rotate=-106] (0,-.2) arc (0:-180: 3.13cm and 1.2cm);

\filldraw   (90:5cm)  --(90:3cm);

\filldraw (90:3cm) -- (70:3cm) node[right]{$x_1$} circle (5pt)  -- (90:5cm)  (70:3cm) -- (30:1cm);

\filldraw   (30:1cm) -- (90:1.5cm) node[right]{$v$} circle (5pt)  --  (150:1cm);
\filldraw   (30:1cm) -- (90:3cm) circle (5pt)  --  (150:1cm) (90:1.5cm)--(90:3cm);

\end{tikzpicture}
\begin{tikzpicture}[scale=.5 ]%9 b)
\filldraw (0,-.2) circle (5pt) node[right]{$x_2$} -- (30:1cm)  circle (5pt)  -- (150:1cm) circle (5pt)  -- (270:1cm) circle (5pt)  (0,-.2)  -- (150:1cm) (270:1cm)node[below]{$x_3$}--(30:1cm) (270:1cm) -- (0,-.2);
\filldraw   (270:1cm) -- (210:2cm) node[left]{$y$} circle (5pt)  --  (150:1cm);
%\filldraw   (30:1cm) -- (330:2cm) circle (5pt)  --  (270:1cm);
\filldraw   (90:5cm) circle (5pt)node[above]{$u$};
\draw[shift= {(30:1cm)},rotate=-80] (0,0) arc (0:180: 2.3cm and 1cm);
\draw[shift= {(150:1cm)},rotate=-100] (0,0) arc (0:-180: 2.3cm and 1cm);
\draw (270:1cm) arc(240:480: 2.2cm and 3.5cm);
%\draw[shift= {(330:2cm)},rotate=-74] (0,0) arc (0:180: 3.13cm and 1.2cm);
\draw[shift= {(210:2cm)},rotate=-106] (0,0) arc (0:-180: 3.13cm and 1.2cm);

\filldraw   (90:5cm)  --(90:3cm);

\filldraw (90:3cm) -- (110:3cm) circle (5pt) node[below]{$x_1$} -- (90:5cm)  (110:3cm) -- (150:1cm);

\filldraw   (30:1cm) -- (90:1.5cm) node[right]{$v$} circle (5pt)  --  (150:1cm);
\filldraw   (30:1cm) -- (90:3cm) circle (5pt)  --  (150:1cm) (90:1.5cm)--(90:3cm);

\end{tikzpicture}
\begin{tikzpicture}[scale=.5 ] %10
\filldraw (0,0)  node[right]{$v$} circle (5pt)  -- (30:1cm)  circle (5pt) -- (150:1cm) circle (5pt)  -- (270:1cm) circle (5pt)  (0,0) -- (150:1cm) (270:1cm)--(30:1cm) (270:1cm) -- (0,0);
\filldraw   (270:1cm) -- (210:2cm) circle (5pt)  --  (150:1cm);
\filldraw   (30:1cm) -- (330:2cm) circle (5pt)  --  (270:1cm);
\filldraw   (90:5cm) circle (5pt);
\draw[shift= {(30:1cm)},rotate=-80] (0,0) arc (0:180: 2.3cm and 1cm);
\draw[shift= {(150:1cm)},rotate=-100] (0,0) arc (0:-180: 2.3cm and 1cm);
\draw (270:1cm) arc(240:480: 2.2cm and 3.5cm);
\draw[shift= {(330:2cm)},rotate=-74] (0,0) arc (0:180: 3.13cm and 1.2cm);
\draw[shift= {(210:2cm)},rotate=-106] (0,0) arc (0:-180: 3.13cm and 1.2cm);

\filldraw   (30:1cm) -- (90:3cm) circle (5pt)  --  (150:1cm);
\filldraw   (90:5cm)  --(90:3cm);

\filldraw (90:3cm) -- (70:3cm) circle (5pt)  -- (90:5cm)  (70:3cm) -- (30:1cm);
\filldraw (90:3cm) -- (110:3cm) circle (5pt) -- (90:5cm) (110:3cm) -- (150:1cm);

\end{tikzpicture}
\begin{tikzpicture}[scale=.5 ]%11
\filldraw (0,0) node[right]{$v$} circle (5pt)  -- (30:1cm)  circle (5pt) -- (150:1cm) circle (5pt)  -- (270:1cm) circle (5pt)  (0,0) -- (150:1cm) (270:1cm)--(30:1cm) (270:1cm) -- (0,0);
\filldraw   (270:1cm) -- (210:2cm) circle (5pt)  --  (150:1cm);
\filldraw   (30:1cm) -- (330:2cm) circle (5pt)  --  (270:1cm);
\filldraw   (90:5cm) circle (5pt);
\draw[shift= {(30:1cm)},rotate=-80] (0,0) arc (0:180: 2.3cm and 1cm);
\draw[shift= {(150:1cm)},rotate=-100] (0,0) arc (0:-180: 2.3cm and 1cm);
\draw (270:1cm) arc(240:480: 2.2cm and 3.5cm);
\draw[shift= {(330:2cm)},rotate=-74] (0,0) arc (0:180: 3.13cm and 1.2cm);
\draw[shift= {(210:2cm)},rotate=-106] (0,0) arc (0:-180: 3.13cm and 1.2cm);

\filldraw   (90:5cm)  --(90:3cm);

\filldraw (90:3cm) -- (70:3cm) circle (5pt)  -- (90:5cm)  (70:3cm) -- (30:1cm);
\filldraw (90:3cm) -- (110:3cm) circle (5pt) -- (90:5cm) (110:3cm) -- (150:1cm);

\filldraw   (30:1cm) -- (90:1.5cm) circle (5pt)  --  (150:1cm);
\filldraw   (30:1cm) -- (90:3cm) circle (5pt)  --  (150:1cm) (90:1.5cm)--(90:3cm);

\end{tikzpicture}

\caption{The two graphs on the left are nine-vertex graphs from Case 2.2.1. 
The two graphs on the right are ten- and eleven-vertex graphs from Case 2.2.2. } 
\label{exc}
\end{figure}

\textbf{Case 2.2.2}
For each $i,j \in\{1,2,3\}$, $i\neq j$ there is at least one vertex in the region $uv_iv_ju$. 

Note that $A_8$ is a subgraph of $G$ in this case.
Thus $n\neq 9$, since there is no vertex of degree $4$ in $G$.
For  $n\in \{10,11\}$ there is a unique graph containing $A_8$ as a subgraph and no vertex of degree four, see the two graphs on the right in Figure~\ref{exc}. 
These graphs contain $110$ and $144$ cycles of length five, respectively.
Note that for $g(11)= 144$ the graph on the right in Figure~\ref{exc} is $A_{11}$, which is another extremal construction distinct from $D_{11}$. 

If $n \geq 12$, then $\abs{X} \geq 4$ since $\abs{Y}\leq 3$. 
By the pigeonhole principle there exists an $i\in \{1,2,3\}$ such that in the region bounded by $uv_iv_{i+1}u$ there are at least two vertices of $X$.
Without loss of generality, we may assume $i=1$.  
Let $x_1\in X$ be the vertex such that $x_1v_1v_2x_1$ is a face, and let $x_2$ be its neighbor from $X$.

Since there is no vertex of degree $4$ in $G'$, either $x_1$ is not adjacent to a vertex from $Y$ or $x_1$ is adjacent to at least two such vertices. 
If $x_1$ is adjacent to two vertices in $Y$, neither of those vertices is adjacent to the vertex $u$. 
Hence both of them are in at most two paths of length three with terminal vertices on the face $v_1v_2v_3v_1$. 
It follows that the number of paths of length three with terminal vertices on the face $v_1v_2v_3v_1$ is at most $4((n-3)-1)+4$ by Lemma~\ref{triangle}, and we are done by induction.
Hence we may assume the neighborhood of $x_1$ is $\{v_1,v_2,x_2\}$. 
Since $G'$ is trianglulated there is a vertex from $X\cup \{u\}$ adjacent to  $v_1$, $v_2$ and $x_2$.
Let $x_3$ be such a vertex.

Consider the vertex $x_1$ of degree $3$ with neighborhood $\{v_1,v_2,x_2\}$. 
Either we are done by Case 2.2.1 for the vertex $x_1$, or both of the bounded regions $v_1x_2x_3v_1$ and $v_2x_2x_3v_2$  contain a vertex. 
Note that in the case these regions contain a vertex, they contain a vertex from $Y$. 
Since the size of $Y$ is at most $3$, either we are done or the bounded regions $v_1v_3uv_1$ and $v_3v_2uv_3$ contain at most one vertex from $X$. 
Hence the number of vertices in $G$ is at most $\abs{X}+\abs{Y}+5\leq 4+3+5=12$. 
Thus if $n>12$ we are done. 
If $n=12$ then $\abs{X}=4$, then $\abs{Y}=3$ and $G$ is obtained from the first graph from right on Figure~\ref{Fig:2b**} by adding a vertex from $Y$, a contradiction since~$G$ has no vertex of degree~$4$. 

\begin{figure}[bt]
\centering
\begin{tikzpicture}[scale=.5 ]%11
\filldraw (0,0) node[right]{$x_1$} circle (5pt)  -- (30:1cm) node[right]{$v_2$} circle (5pt) -- (150:1cm) node[left]{$v_1$} circle (5pt)  -- (270:1cm) node[below]{$x_2$} circle (5pt)  (0,0) -- (150:1cm) (270:1cm)--(30:1cm) (270:1cm) -- (0,0);
\filldraw   (270:1cm) -- (210:2cm) node[below]{$y_2$} circle (5pt)  --  (150:1cm);
\filldraw   (30:1cm) --  (330:2cm) node[below]{$y_1$} circle (5pt)  --  (270:1cm);
\filldraw   (90:5cm) node[above]{$u$} circle (5pt);
\draw[shift= {(30:1cm)},rotate=-80] (0,0) arc (0:180: 2.3cm and 1cm);
\draw[shift= {(150:1cm)},rotate=-100]   (0,0) arc (0:-180: 2.3cm and 1cm);
\draw (270:1cm) arc(240:480: 2.2cm and 3.5cm);
\draw[shift= {(330:2cm)},rotate=-74] (0,0) arc (0:180: 3.13cm and 1.2cm);
\draw[shift= {(210:2cm)},rotate=-106] (0,0) arc (0:-180: 3.13cm and 1.2cm);

\filldraw   (90:5cm)  -- node[below]{$v_3$} (90:3cm);

\filldraw (90:3cm) -- (70:3cm) node[below]{$x_4$} circle (5pt)  -- (90:5cm)  (70:3cm) --  (30:1cm);
\filldraw (90:3cm) --  (110:3cm)  node[below]{$x_3$} circle (5pt) --  (90:5cm) (110:3cm) --  (150:1cm);

\filldraw    (30:1cm) -- (90:1.5cm) circle (5pt)  --  (150:1cm);
\filldraw   (30:1cm) -- node[below]{$v$} (90:3cm) circle (5pt)  --  (150:1cm) (90:1.5cm)--(90:3cm);

\end{tikzpicture}

\caption{The eleven vertex subgraph of $G$ from Case 2.2.2.} 
\label{Fig:2b**}
\end{figure}

\textbf{Case 3.}
There is no vertex of degree $3$ or $4$, i.e.,the minimum degree of $G$ is $5$.

Note that this implies that $n\geq 12$.
Let $v$ be a vertex of degree~$5$, and let $v_1,v_2,\dots,v_5$ be the neighbors of $v$ arranged such that $v_1v_2v_3v_4v_5v_1$ is a cycle. 
Since $G$ is planar there exists an $i\in [5]$ such that  $\{v_i,v_{i-2}\}\notin E(G)$ and $\{v_i,v_{i+2}\}\notin E(G)$.
Without loss of generality, we assume $i=1$. 
Moreover, since $G$ is $K_5$-minor-free either $\{v_2,v_4\}$ or $\{v_3,v_5\}$ is not an edge of $G$. 
We may assume, without loss of generality that $\{v_3,v_5\}$ is not an edge in $G$.
Consider the graph $G'$ obtained from $G$ by removing the vertex $v$ with all edges incident to it and adding the edges $\{v_1,v_3\}$ and $\{v_1,v_4\}$.

As in Case 1, we say that a $5$-cycle in $G$ containing $v$ is \emph{type-$k$}, for $k=0,1,2$, if it contains $v$ and precisely $k$ vertices not in $N[v]$, and we say that a five-cycle in $G'$ is \emph{type-$k$}, for $k=0,1,2$, if it contains at least one of the edges $\{v_1,v_3\}$ or $\{v_1,v_4\}$ and precisely $k$ vertices not in $N(v)$. 
Let $\N_k$ be the number of type-$k$ cycles in $G$ minus the number of type-$k$ cycles in $G'$. 
We say that a path is \emph{internal} if all of its vertices are in $N[v]$, and we say that a path is \emph{external} if all of its non-terminal vertices are outside $N[v]$. As before $\N(C_5,G)-\N(C_5,G') = \N_0 + \N_1 + \N_2 + \N_3 \leq \N_0 + \N_1 + \N_2.$

We have
\[
   \N_1= \sum_{1\leq i < j \leq 5}h(i,j)\abs{(N(v_i)\cap N(v_j)\setminus N[v]},
\]
where $h(i,j)$ is the number of internal paths of length $3$ from $v_i$ to $v_j$ which contain $v$ in $G$ minus the number of  paths of length $3$ from $v_i$ to $v_j$ containing at least one of the edges $\{v_1,v_3\}$ or $\{v_1,v_4\}$ in $G'$. 
Note that $G[N[v]]$ contains the five-cycle $v_1v_2v_3v_4v_5v_1$ and may contain the edges $\{v_2,v_4\}$ or $\{v_2,v_5\}$ the other edges are  not present in  $G[N[v]]$.

Here we consider internal paths of length $3$ from $v_i$ to $v_j$ incident with $v$ in $G$ containing neither the edge  $\{v_2,v_4\}$ nor  $\{v_2,v_5\}$.  If $j=i+1$ then there are two such paths: $v_iv v_{i+2}v_{i+1} $ and $v_i v_{i-1}vv_{i+1}$. 
While for $j=i+2$  there are four such paths:  $v_{i+2}vv_{i-1}v_i, v_{i+2}vv_{i+1}v_i, v_{i+2}v_{i+1}vv_i$ and $v_{i+2}v_{i+3}vv_i$. 

Now we list the internal paths of length $3$ from $v_i$ to $v_j$ in $G'$ containing at least one of the edges $\{v_1,v_3\}$, $\{v_1,v_4\}$ and containing neither the edge  $\{v_2,v_4\}$ nor  $\{v_2,v_5\}$:  $v_1v_4v_3v_2$, $v_1v_3v_4v_5$, $v_3v_1v_5v_4$, $v_3v_2v_1v_4$, $v_2v_1v_4v_3$, $v_5v_1v_3v_4$, $v_2v_1v_3v_4$, $v_2v_3v_1v_4$, $v_5v_1v_4v_3$, $v_5v_4v_1v_3$, $v_2v_1v_4v_5$, $v_2v_3v_1v_5$, $v_1v_3v_4v_5$. 

Next we consider the internal paths of length $3$ from $v_i$ to $v_j$, $i,j\in [5]$, incident with $v$ in $G$  containing either $\{v_2,v_4\}$ or  $\{v_2,v_5\}$.
First note that such paths contain exactly one of the edges $\{v_2,v_4\}$ or  $\{v_2,v_5\}$. In what follows we list all such paths and transform them, with the exception of $v_5 v v_2 v_4$, to internal paths of length $3$ from $v_i$ to $v_j$  in $G'$ containing at least one of the edges $\{v_1,v_3\}$, $\{v_1,v_4\}$ as follows
(see the first two graphs on Figure~\ref{d=5}). 
% The path $v_5 v v_2 v_4$ will not be mapped to a path in $G'$.
\begin{align*}
    v_1 v v_2v_4\rightarrow v_1 v_3 v_2 v_4,    \\v_3 v v_2v_4\rightarrow v_3 v_1 v_2 v_4,
    % \[v_5 v v_2v_4\rightarrow \mbox{-},\]
\\v_1 v v_4 v_2\rightarrow v_1 v_3 v_4 v_2,
\\v_3 v v_4 v_2\rightarrow v_3 v_1 v_4 v_2,
\\v_5 v v_4 v_2\rightarrow v_5 v_1 v_4 v_2,
\\v_1 v v_2 v_5\rightarrow v_1 v_3 v_2 v_5,
\\v_3 v v_2 v_5\rightarrow v_3 v_1 v_2 v_5,
\\v_4 v v_2 v_5\rightarrow v_4 v_1 v_2 v_5,
\\v_1 v v_5v_2\rightarrow v_1 v_4 v_5v_2,
\\v_3 v v_5v_2 \rightarrow v_3 v_1 v_5v_2,
\\ v_4 v v_5v_2 \rightarrow v_4 v_1 v_5v_2.
\end{align*}

% \[v_1 v v_2v_4\rightarrow v_1 v_3 v_2 v_4,\]        \[v_3 v v_2v_4\rightarrow v_3 v_1 v_2 v_4,\]
% % \[v_5 v v_2v_4\rightarrow \mbox{-},\]
% \[v_1 v v_4 v_2\rightarrow v_1 v_3 v_4 v_2,\]
% \[v_3 v v_4 v_2\rightarrow v_3 v_1 v_4 v_2,\]
% \[v_5 v v_4 v_2\rightarrow v_5 v_1 v_4 v_2,\]
% \[v_1 v v_2 v_5\rightarrow v_1 v_3 v_2 v_5,\]
% \[v_3 v v_2 v_5\rightarrow v_3 v_1 v_2 v_5,\]
% \[v_4 v v_2 v_5\rightarrow v_4 v_1 v_2 v_5,\]
% \[v_1 v v_5v_2\rightarrow v_1 v_4 v_5v_2,\]
% \[v_3 v v_5v_2 \rightarrow v_3 v_1 v_5v_2,\]
% \[ v_4 v v_5v_2 \rightarrow v_4 v_1 v_5v_2.\] 
Note that for each path the corresponding path has the same terminal vertices. Even more each edge $\{v_2,v_4\}$ and  $\{v_2,v_5\}$ is either in both the internal path in $G$ and the corresponding internal path in $G'$ or neither. 
Finally, we can determine the function $h(i,j)$ (recall that  $h(i,j)$ is the number of internal paths of length $3$ from $v_i$ to $v_j$ which contain $v$ in $G$ minus the number of  paths of length $3$ from $v_i$ to $v_j$ containing at least one of the edges $\{v_1,v_3\}$ or $\{v_1,v_4\}$ in $G'$).
Note that  $h(4,5) = 1$ if $\{v_2,v_4\}$ is not an edge of $G$, and $h(4,5) = 2$ otherwise.  The values of $h(i,j)$ are recorded in Figure~\ref{chart}.
\begin{figure}[h]   
\centering
%-------------
$\begin{array}{|c|c|c|c|c|c|c|c|c|c|c|}
\hline\{i,j\} & \{1,2\} &  \{1,3\} & \{1,4\} & \{1,5\} & \{2,3\} & \{2,4\} & \{2,5\} & \{3,4\} & \{3,5\} & \{4,5\} \\
\hline
h(i,j) & 1 & 4 & 4 & 1 & 1 & 2 & 2 & 0 & 2 & 1 ~ \mbox{or} ~ 2\\
\hline
\end{array}$
\caption{A table of the values for $h(i,j)$ for all possible pairs $(i,j)$.}
\label{chart}
\end{figure}

Let us consider the type-$2$ five-cycles. 
Note that for each path of length $3$ $v_ixyv_j$ such that $i\neq j$, $x\neq y$ and  $x,y \not \in N[v]$, there is exactly one type-$2$ five-cycle containing the path and $v$, namely $vv_ixyv_jv$.
While in $G'$ we have that for $\{i,j\} \not = \{2,5\}$, there is  one type-$2$  five-cycle using at least one of the edges $\{v_1,v_3\}$ or  $\{v_1,v_4\}$. 
The cycles are the following: if $1 \not \in \{i,j\}$ and $\{i,j\} \not = \{2,5\}$ the type-$2$ $5$-cycle is $v_ixyv_jv_1v_i$ which contains at least one of the edges $\{v_1,v_3\}$ or  $\{v_1,v_4\}$ (see Figure~\ref{d=5}). 
If $i = 1$ and $j \in \{2,4\}$ the 5-cycle $v_1xyv_jv_3v_1$ contains the edge $\{v_3,v_1\}$. 
Finally if $i=1$ and $j\in \{3,5\}$ the 5-cycle $v_1xyv_jv_4v_1$ contains the edge $\{v_1,v_4\}$.

Let us now consider type-$0$ $5$-cycles. 
Observe that there are five type-$0$ five-cycles in $G$ if  $G$  contains neither of the edges $\{v_2,v_5\}$ and $\{v_2,v_4\}$.
There are ten type-$0$ five-cycles in $G$ when $G$ contains exactly one of the edges $\{v_2,v_5\}$ or $\{v_2,v_4\}$.
There are seventeen type-$0$ five-cycles in $G$ if  $G$ contains both of the edges $\{v_2,v_5\}$ and $\{v_2,v_4\}$.
 On the other hand, in $G'$ the number of type-$0$ five-cycles is zero if  $G$ does not contain the edges $\{v_2,v_5\}$ and $\{v_2,v_4\}$. 
There is one  type-$0$ $5$-cycle in $G'$ if $\{v_4,v_2\}$ is an edge and three  type-$0$ $5$-cycles in $G'$ if $\{v_2,v_5\}$ is an edge.  There are five type-$0$ $5$-cycles in $G'$ if  $G$ contains  the edges $\{v_2,v_5\}$ and $\{v_2,v_4\}$. 
 Hence we have 
 \[\N_0 \in \{5,9,7,12\}.\] 

\begin{figure}[t]
\centering
\begin{tikzpicture}
\draw[red,dashed] (306:1cm)--(90:1cm);
% --(234:1cm)
\draw[blue] (18:1cm) arc(0:180: .96cm and 1.2cm);
\draw[blue] (-54:1cm) --  (-18:2cm) -- (18:1cm);
\draw[blue,dashed] (162:1cm) -- (0,0) -- (306:1cm);
\draw[red] (162:1cm) -- (90:1cm);

\filldraw (-18:2cm) circle (2pt);

\filldraw (90:1cm) circle (2pt) node[above]{$v_1$}  (162:1cm) circle (2pt) node[left]{$v_2$}  -- (234:1cm) circle (2pt) node[left]{$v_3$} -- (306:1cm) circle (2pt) node[below]{$v_4$} -- (18:1cm) circle (2pt) node[right]{$v_5$} -- (90:1cm);

%\filldraw (0,0) circle (2pt) node[below]{$v$}  -- (90:1cm) (0,0) -- (162:1cm)(0,0)  -- (234:1cm) (0,0) -- (306:1cm) (0,0) -- (18:1cm);  
\filldraw (0,0) circle (2pt) node[below]{$v$};
\end{tikzpicture}\quad
\begin{tikzpicture}

\draw[blue,thick]  (162:1cm) -- (-2,.3) -- (-1.8,-.6) -- (234:1cm) (18:1cm) -- (.5,1.5) --  (-.5,1.5) -- (162:1cm) ;
\draw[blue,dashed] (162:1cm) -- (0,0) -- (234:1cm)  (0,0) -- (18:1cm);
\draw[red,thick] (90:1cm)  -- (162:1cm) circle (2pt);
\draw[red,dashed] (90:1cm)--(234:1cm);
\node at (-1.1,.6) {$v_2$};
%(306:1cm)--
\filldraw (90:1cm) circle (2pt) node[above]{$v_1$}  (162:1cm) circle (2pt)  -- (234:1cm) circle (2pt) node[below]{$v_3$} -- (306:1cm) circle (2pt) node[right]{$v_4$} -- (18:1cm) circle (2pt) node[right]{$v_5$} -- (90:1cm);

%\filldraw (0,0) circle (2pt) node[below]{$v$}  -- (90:1cm) (0,0) -- (162:1cm)(0,0)  -- (234:1cm) (0,0) -- (306:1cm) (0,0) -- (18:1cm);  
\filldraw (-2,.3) circle (2pt)  (-1.8,-.6) circle (2pt) (0,0) circle (2pt) node[below]{$v$} (.5,1.5) circle (2pt)  (-.5,1.5) circle (2pt);
\end{tikzpicture}\quad
\caption{In the first graph we have a type-$1$ five-cycle in $G$ using the edge $\{v_2,v_5\}$ and the corresponding cycle in $G'$.  The second graph pictured represents the type-$2$ five-cycles in $G$ and in $G'$.} 
\label{d=5}
\end{figure}

\textbf{Case 3.1.} There is a vertex $u \not = v$, which is adjacent to every vertex of $N(v)$.

Note that in this case, $G$ does not contain the edges  $\{v_2,v_5\}$ and $\{v_2,v_4\}$, thus $\N_0 = 5$. 
For any two consecutive regions,   $uv_{i-1}v_iu$ and $uv_iv_{i+1}u$, one of them must contain a vertex in its interior, since the degree of  $v_i$ is at least five.
It follows that at least three of the regions $uv_iv_{i+1}u$ contain a vertex in their interior. 
However since the minimum degree of the graph is $5$, it follows that each such region must contain at least three vertices (since the vertex in the interior has at most $3$ edges to the vertices of the triangle $uv_iv_{i+1}u$), therefore $n \geq 7+3\cdot 3 = 16$.

We have that $u$ contributes to $18$ type-$1$ five-cycles, and every other vertex is adjacent to at most two vertices of $N(v)$, which must be consecutive since $G$ is planar. 
In particular any other vertex contributes to at most two type-$1$ five-cycles, so $\N_1 \leq 18 + 2(n-7)$.

Finally, we have that each of the vertices from $V(G)\setminus (N[v]\cup \{u\})$ can be in at most one path of length three from $v_2$ to $v_5$, since each such path must include $u$, so $\N_2 \leq (n-7)$. Hence \[\N(C_5,G) - \N(C_5,G') \leq 23 + 3(n-7)<4(n-3),\] 
since $n\geq 16$, and we are done by Equation~\ref{Equation:Induction_second}.

\textbf{Case 3.2.} There exists a vertex $u \in V(G)$ which is adjacent to precisely four vertices of $N(v)$.

By relabeling the vertices we may assume that $v_3$ is not adjacent to $u$.  
Since 
\[
\sum_{\substack{1\leq i < j \leq 5 \\ i,j\neq 3}} h(i,j)\leq 12,
\]
 the vertex $u$ contributes at most $12$ to the value of  $\N_1$. 

Since $G$ is planar it does not contain the  edge  $\{v_2,v_5\}$, thus $\N_0 \leq 9$. 

Every vertex in $V(G)\setminus (N[v]\cup \{u\})$ is in at most one path of length $3$ from $v_2$ to $v_5$, since it must contain the vertex $u$.  
Thus we have $\N_2 \leq (n-7)$.

Here we split the case into two cases depending on whether there is a vertex $w\neq u$ from $V(G)\setminus N[v]$ adjacent to at least three vertices from $N(v)$ or not. 
If such a vertex exists then it is adjacent to $v_2,v_3$ and $v_4$, and it is unique. 
Since the degree of $v_3$ is at least five, one of the regions $wv_4v_3w$ and  $wv_3v_2w$ is not empty. 
Thus this region must contain at least three vertices.
Even more, one of the regions $uv_4v_5u$, $uv_5v_1u$ and $uv_1v_2u$ contains at least three vertices. 
 Hence the number of vertices is at least $8+3+3=14$.  

Since 
\[
\sum_{\substack{2\leq i < j \leq 4 \\ i,j\neq 3}} h(i,j)=3,
\]
 the vertex $w$ contributes at most $12$ to the value of  $\N_1$. 
Every  vertex from  $V(G)\setminus N[v] \setminus \{w,u\}$ has at most two neighbors in $N(v)$, and since $h(i,j)\leq 2$, they contribute  at most two 
to the value of  $\N_1$.  
Thus $\N_1 \leq  12 + 3 + 2(n-8)$ and so \[\N(C_5,G) - \N(C_5,G') \leq  9+12+3+2(n-8)+(n-7) < 4(n-3),\]
since $n\geq 12$, and we are done by Equation~\ref{Equation:Induction_second}.

If no vertex from  from $V(G)\setminus N[v]$ distinct from  $u$ is adjacent to at least three vertices from $N(v)$, then any vertex $x \in N(v_3)\setminus N[v]$ contributes to at most one type-$1$ five-cycle, since $h(2,3),h(3,4)\leq 1$. 
Since $d(v_3)\geq 5$ and $v_3$ is not adjacent to $v_1$ or $v_5$, there are at least two vertices in $N(v_3)\setminus N[v]$. 
Thus we have that $\N_1 \leq 12+2 + 2(n-9)$, so \[\N(C_5,G) - \N(C_5,G') \leq 9+14+2(n-9) +(n-7)< 4(n-3),\] since $n\geq 12$, and we are done by Equation~\ref{Equation:Induction_second}.

\textbf{Case 3.3} Every vertex in $V(G) \setminus N[v]$ is adjacent to at most $3$ vertices from $N(v)$. 

Recall that $v_1$ was chosen to be a vertex such that $\{v_1,v_3\}$ and  $\{v_1,v_4\}$ are not edges of $G$.
Even more, without loss of generality, we may assume that $v_1$ is a vertex such that there is no external path of length two from $v_1$ to $v_3$ or $v_4$. 
Indeed, if there was such a vertex in $V(G)\setminus N[v]$ adjacent to $v_i$ and $v_{i+2}$, then $v_{i+1}$ is a possible choice for $v_1$, otherwise, $v_1$ satisfies all the desired conditions. From here we consider two cases depending on whether or not $\{v_2,v_4\}$ is an edge in $G$.

If $\{v_2,v_4\}$  is an edge of $G$, then we have that no vertex inside the region $v_2v_3v_4v_2$ can be in an external path of length $3$ from $v_2$ to $v_5$. 
Even more this region must contain at least $3$ vertices since $\delta(G)=5$. 
By Lemma~\ref{2(k-1)} there are at most  $2(n-10)$ external path of length $3$ from $v_2$ to $v_5$.  
As we have seen at the beginning of Case~3, $\N_2$ is equal to the number of external paths of length $3$ from $v_2$ to $v_5$. Hence we have $\N_2\leq 2(n-10)$.
There is at most one vertex distinct from $v$ adjacent to $v_2,v_4$, and $v_5$. 
Such a vertex  contributes $6=h(2,4)+h(2,5)+h(4,5)$ to the value of $\N_1$. 
Also, there is at most one vertex distinct from $v$ which is adjacent to $v_2,v_3$ and $v_4$. 
Such a vertex  contributes $3=h(2,3)+h(2,4)+h(3,4)$ to the value of $\N_1$. 
Every other vertex contributes at most two to the value of  $\N_1$. 
Therefore $\N_1 \leq 9 + 2(n-8)$, and since $\N_0 \leq 12$ we have 
\[\N(C_5,G) - \N(C_5,G') \leq 12+9+2(n-8)+2(n-10) < 4(n-3),\]
and we are done by Equation~\ref{Equation:Induction_second}.

If $\{v_2,v_4\}$ is not an edge of $G$, then $\N_0 \leq 7$. 
By Lemma~\ref{2(k-1)}, we have that the number of paths of length $3$ from $v_2$ to $v_5$ is at most $2(n-7)$, hence  $\N_2 \leq 2(n-7)$. 
Each vertex from $V(G)\setminus N[v]$ with exactly two neighbors in $N(v)$ contributes at most two to $\N_1$, since $h(i,j)\leq 2$ for every $0<i<j<5$, $\{i,j\}\neq \{1,3\},\{1,5\}$. 
Consider vertices from  $V(G)\setminus N[v]$ with at least  three neighbors from $N(v)$. There are at most two such vertices, and none of them is adjacent to $v_1$. 
Since $G$ is planar, they are adjacent to $v_2,v_3,v_5$ and $v_3,v_4,v_5$ or $v_2,v_4,v_5$ and $v_2,v_3,v_4$, respectively.
Therefore $\N_1 \leq h(2,3)+h(2,5)+h(3,5)+h(3,4)+h(3,5)+h(4,5)+2(n-8)=8+2(n-8).$
Finally we have \[\N(C_5,G) - \N(C_5,G') \leq 7 +8+2(n-8)+ 2(n-7) < 4(n-3),\]
since $n\geq 12$ and and we are done by Equation~\ref{Equation:Induction_second}.
\end{proof}

\section*{Acknowledgements}
We thank the referees for their many useful comments, which improved the presentation of the manuscript.  The research of the first, third and fifth authors is partially supported by the National Research, Development and Innovation Office -- NKFIH, grant K116769, K132696 and SNN117879.  
The research of the third author was supported by the Institute for Basic Science (IBS-R029-C4).
The research of the fourth author is supported by the Institute for Basic Science (IBS-R029-C1).

\end{document}

%% file: PGProperties.tex
% A graph is said to be \textit{planar}, if it can be drawn in the plane so that its edges intersect only at their ends. 
% Such a drawing of a planar graph $G$ is called a \textit{planar embedding} of $G$. 

% Let $C=x_1,x_2,\dots x_k,x_1$ be a cycle in $G$, then $C$ is said to \textit{separate} vertices $y,z \in V(G)$ in a planar embedding of $G$ if one of $y$ or $z$ is in the interior of the curve formed by embedding of the cycle and the other one is in the exterior. 

% One of the most basic results in planar graphs is Euler's Formula.

% \begin{theorem}
% If $G$ is a connected planar graph, then \begin{displaymath}
% v(G)-e(G)+f=2,
% \end{displaymath}
% where $f$ is the number of faces the planes is divided into in a planar embedding of $G$.
% \end{theorem}

% From Euler's formula we obtain three important corollaries.

% \begin{corollary}
% Every planar embedding of a planar graph has the same number of faces. 
% \end{corollary}

% \begin{corollary}
% If $G$ is a planar graph with $n \geq 3$ vertices, then $e(G) \leq 3n - 6.$ 
% \end{corollary}

% As a consequence of Euler formula, it follows that a if $G$ is an $n$-vertex planar graph, then $G$ is maximal if

An $n$-vertex planar graph $G$ has $3n-6$ edges if and only if in any planar embedding every face of $G$ is a triangle, and so an $n$-vertex planar graph with $3n-6$ edges is called a triangulation. 
A planar graph $G$ is called maximal if it is not possible to add an extra edge to $G$ and preserve the planarity. 
Moreover, if an $n$-vertex graph has less than $3n-6$ edges, then it is always possible to add another edge while keeping the graph planar. Hence the maximal planar graphs are precisely the triangulations.

% A simple but important corollary is that the maximum 

% \begin{corollary}
% If $G$ is a planar graph, then $\delta(G) \leq 5$. And If $G$ is a maximal planar graph, then $\delta(G)\geq 3.$
% \end{corollary}

Recall that by Kuratowski's Theorem~\cite{kur} a graph $G$ is planar if and only if $G$ does not contain a $K_5$ or $K_{3,3}$ subdivision. 

% It is known that $K_5$ and $K_{3,3}$ are not planar graphs, and therefore, a graph $G$ can only be planar if it is $K_{5}$-free and $K_{3,3}$-free.
% \begin{theorem}[Kuratowski \cite{kur}]
% \label{Kuratowski}
% A graph $G$ is planar if and only if  $G$ does not contain a subdivision of $K_5$ or $K_{3,3}$ as a subgraph.
% \end{theorem}

Recall that for a given graph $G$, if $e = \{v,u\}$ is an edge of $G$, then the \textit{contraction} of the edge $e$ in $G$ is the graph obtained from $G$ by replacing the two vertices $\{v,u\}$ with a new vertex $w$ and replacing the edges of the form $\{v,x\}$ and $\{y,u\}$ with the edges $\{w,x\}$ and $\{y,w\}$ taking the new edges without multiplicity. It is known that contracting an edge from a planar graph would result in a planar graph.

% \begin{claim}\label{ncycle}
% If $G$ is a maximal planar graph on $n\geq 4$ vertices, then for every any vertex $v$ of $G$, the graph induced by $N(v)$ is Hamiltonian.
% \end{claim}

% \begin{proof}
% Consider a planar embedding of $G$, and delete every edge of $v$ to obtain $G'$, then in this planar embedding, we have that $v$ is inside a face of $G'$, say $v$ is inside a face $C= v_1v_2\dots v_nv_1$, we will show that $N(v) = \{v_1,v_2,\dots,v_n\}$. 
% It is clear that $v$ cannot have any neighbor outside of $C$, since $C$ separates $v$ from the other vertices of $G$, and since the face $C$ should be triangulated in the embedding of $G$ then $v$ must be adjacent to every vertex of $C$. 
% \end{proof}

We will give three lemmas which will be used later in the proof of Theorem~\ref{c5}.

\begin{lemma}\label{pforest}
Let $G$ be a planar graph and $\{u,v\}$ be an edge, then $G[N(u)\cap N(v)]$ is a path forest. Moreover, the graph induced by $\{u,v\} \cup \left(N(u) \cap N(v)\right)$ is a triangulation if and only if $G[N(u) \cap N(v)]$ is a path.
\end{lemma}
\begin{proof}
First, we show that $G'=G[N(u)\cap N(v)]$ is acyclic.
Suppose by contradiction, there is a cycle $C$ in the common neighborhood  $N(u)\cap N(v)$, then 
 we have a $K_5$-minor. 
 Indeed, if we contract this cycle to a triangle, then this triangle together with $v$ and $u$ forms a $K_5$, contradicting  planarity. 
 Hence $G'$ is acyclic. 
 
 Now we show that $d_{G'}(w) \leq 2$ for all $w \in V(G')$.  
 Suppose $d_{G'}(w)\geq 3$ for some $w \in V(G')$, then taking three distinct vertices $w_1,w_2,w_3\in N(u)\cap N(v) \cap N(w)$ yields a $K_{3,3}$,  a contradiction to planarity. 
 Therefore, $d_{G'}(w) \leq 2$ for all $w \in V(G')$, hence $G'$ is a path forest. 
 
The graph induced by the vertex set $\{v,u\} \cup \left(N(u) \cap N(v)\right)$ is a triangulation if and only if it has $3\abs{N(u)\cap N(v)}$ edges (here we use the basic fact that the planar graphs with the maximum number of edges are triangulations).
There are $2\abs{N(u)\cap N(v)}+1$ edges incident with $u$ or $v$. 
Thus, the graph induced by $\{v,u\} \cup \left(N(u) \cap N(v)\right)$ is a triangulation if and only if  the graph induced by $N(u)\cap N(v)$  has precisely $\abs{N(u) \cap N(v)} -1$ edges. 
Therefore, it is a triangulation if and only if $G[N(u)\cap N(v)]$ is a path.
\end{proof}

\begin{lemma}\label{2(k-1)}
Let $G$ be a planar graph on $k \ge 3$ vertices and let $\{u,v\}$ be an edge of $G$, then the number of paths of length~$3$ from $u$ to $v$ is at most $2(k-3)$.
\end{lemma}

\begin{proof}
We may assume that $N(u)\cap N(v) \neq V(G)\setminus\{u,v\}$,  since otherwise Lemma~\ref{pforest} would imply the result. 
Indeed since $G[N(u)\cap N(v)]$ is a path forest, it has at most $k-3$ edges and so there are at most $2(k-3)$ paths of length~$3$ between $u$ and $v$.

We are going to prove the lemma by induction on $k$, the bound trivially holds for $k=3$. Consider the set of ordered pairs $A = \{(x,y):uxyv$ is a path$\}$. %If there is a vertex $z$ that is in at most two pairs from $A$, then after removing this vertex the result will follow by induction. If every vertex $z$ is in at least three pairs from $A$, then
Let $x$ be a vertex that is not in $N(u)\cap N(v)$, and without loss of generality suppose $x\not \in N(v)$. If $x$ is in at most two pairs from $A$, then we can remove $x$ and we would be done by induction, so suppose $x$ is in at least three pairs from $A$. 
It follows that there exist three vertices distinct from $u$, say $y_1, y_2$ and $y_3$, in $N(x)\cap N(v)$. 
Suppose further that $y_2$ is such that the cycle $xy_1vy_3x$ separates $y_2$ and $u$ (see Figure \ref{le}).  Then $y_2$ is not adjacent to $u$. Now contract the edge $\{x,y_2\}$ to a vertex $x'$, and note that 
%the only path of length three which is deleted is $vxy_2u$
 with the exception of  $uxy_2v$, every path of length~$3$ from $u$ to $v$ using either $x$ or $y_2$ yields a unique path of length~$3$ from $u$ to $v$ containing $x'$ after the contraction.
So, by induction, in the contracted graph we have at most $2(k-4)$ paths of length~$3$ from $u$ to $v$, therefore in the original graph we have at most  $2(k-4) + 1 < 2(k-3)$ such paths.
\end{proof}
\begin{figure}   
    \centering
\begin{tikzpicture}[scale=1.2]
\filldraw (0,0) node[below]{$u$} circle (2pt) (2,0) node[below]{$v$} circle (2pt)  (.5,1)node[above]{$x$}   circle (2pt) (1,.5) node[left]{$y_1$}  circle (2pt)  (1.5,1) node[right]{$y_2$}  circle (2pt) (2,1.5) node[right]{$y_3$}  circle (2pt);
\draw (0,0) -- (.5,1) -- (1,.5) -- (2,0) (.5,1) -- (1.5,1) -- (2,0) (.5,1) -- (2,1.5) -- (2,0) -- (0,0) ; 
\end{tikzpicture}
    \caption{Separating cycle from proof of Lemma~\ref{2(k-1)}.}
    \label{le}
\end{figure}

\begin{lemma} 
\label{triangle} Let $k \ge 4$, $G$ be a $k$-vertex planar graph and let $T$ be a set of three vertices of any triangular face of $G$, then the number of paths of length~$3$ with both terminal vertices in $T$ is at most $4(k-1)$. If there is no vertex adjacent to all vertices in $T$, then there are at most $4k-9$ such paths. 
\end{lemma}

\begin{proof}
First, we will prove the lemma holds in the case where every vertex of $G$ is adjacent to at least two vertices of $T$. 
Let $x_1,x_2,x_3$ be the vertices of $T$, and let $A = N(x_1)\cap N(x_2)\setminus\{x_3\}$, $B = N(x_2)\cap N(x_3)\setminus\{x_1\}$ and $C = N(x_3)\cap N(x_1)\setminus\{x_2\}$.  
Note that there is at most one vertex in the intersection $N(x_1) \cap N(x_2) \cap N(x_3)$. Indeed, if there were two such vertices, and we added another vertex to the interior of $T$ and connected it to each vertex of $T$, then the resulting graph  would be a planar graph containing a $K_{3,3}$, a contradiction.

Note that the only vertices in $N(x_1)\cup N(x_2) \cup N(x_3)$ which are not in $A\cup B\cup C$ are the vertices of $T$. 
Thus $\abs{A}+\abs{B}+\abs{C}\leq k-3+2\leq k-1$, since the $A\cup B \cup C = V(G)\setminus T$ and at most one vertex is in  $A \cap B\cap C$. 

We now distinguish two cases depending on whether  there is a vertex adjacent to the three vertices of $T$.

\begin{figure}[bht]
\centering
\begin{tikzpicture}[scale=.6 ] %8
\draw[thick,red] (210:2cm) -- (0,2) --  (.9,2.2) -- (330:2cm);
\draw[shift= {(210:2cm)},rotate=-106,thick,blue] (0,0) arc (0:-180: 3.13cm and 1.2cm);
\draw[thick,blue,smooth] (1.3,3.5) -- (0,5) (1.3,3.5)  -- (330:2cm);
% \filldraw (0,0) circle (5pt)  -- (30:1cm)  circle (5pt) -- (150:1cm) circle (5pt)  -- (270:1cm) circle (5pt)  (0,0) -- (150:1cm) (270:1cm)--(30:1cm) (270:1cm) -- (0,0);
% \filldraw   (90:2cm)  node[left]{$x_3$} circle (3pt)  -- (330:2cm)  node[below]{$x_2$}circle (3pt) -- (210:2cm) node[below]{$x_1$} circle (3pt) -- (90:2cm)  ;
\filldraw   (90:2cm)  node[left]{$x_3$} circle (3pt)  -- (330:2cm)  node[below]{$x_2$}circle (3pt) -- (210:2cm) node[below]{$x_1$} circle (3pt)  ;
%\filldraw (90:2cm)    --  (.9,2.2)  circle (3pt) -- (330:2cm);
\filldraw (.9,2.2) circle (3pt);
\filldraw (90:2cm)    --  (1.1,2.8)  circle (3pt) -- (330:2cm);
% \filldraw (90:2cm)    --  (1.3,3.5)  circle (3pt) -- (330:2cm);
 \filldraw (90:2cm)    --  (1.3,3.5)  circle (3pt) ;
% \filldraw   (30:1cm) --   --  (270:1cm);
% \draw[shift= {(30:1cm)},rotate=-80] (0,0) arc (0:180: 2.3cm and 1cm);
% \draw[shift= {(150:1cm)},rotate=-100] (0,0) arc (0:-180: 2.3cm and 1cm);
% \draw (270:1cm) arc(240:480: 2.2cm and 3.5cm);
\draw[shift= {(330:2cm)},rotate=-74] (0,0) arc (0:180: 3.13cm and 1.2cm);

\filldraw   (90:5cm) node[above]{$u$} circle (3pt) -- (90:2cm);
\node[] at (270:1.7) {$A$};   
\node[] at (60:1.7) {$B$};
\node[] at (120:1.7) {$C$};
\end{tikzpicture}
\caption{Two possible paths of length~$3$ from $x_1$ to $x_2$.}
\label{Figure:Face_T}
\end{figure}

% \textcolor{blue}{First Case here}
\textbf{Case 1.}
There is a vertex $u$ adjacent to every vertex of $T$, see Figure~\ref{Figure:Face_T}.
The vertex $u$ can have at most one neighbor in each of the sets $A$, $B$ and $C$. Suppose not, if $u$ had two neighbors in the same set, say $y_1,y_2 \in A$, then each vertex from $\{x_1,y_1,y_2\}$ is adjacent to every vertex of $\{u,x_2,x_3\}$, which is a contradiction to the planarity of $G$. 
The number of paths of length~$3$ from $x_1$ to $x_2$ using $x_3$ is precisely $|B|+|C|$, since each such path is of the form $x_1x_3bx_2$ with $b\in B$ or $x_1cx_3x_2$ with $c\in C$.  
There are at most two paths of length~$3$ from $x_1$ to $x_2$ using $u$ and a vertex not in $A$ since $u$ has at most one neighbor in $B$ and at most one neighbor in $C$ and we are under the assumption that  every vertex of $G$ is adjacent to at least two vertices of $T$. 
Since, there is no edge between $A$ and $B$, $B$ and $C$ or $C$ and $A$, then all other paths of length~$3$ from $x_1$ to $x_2$ have internal vertices only from $A$, inducing an edge. There are at most $|A|-1$ edges in $G[A]$ by Lemma \ref{pforest}, hence we have at most $2|A|-2$ such paths. 
Thus, we have at most $2|A|+|B|+|C|$ paths of length~$3$ from $x_1$ to $x_2$. Similarly we have at most  $|A|+2|B|+|C|$ paths of length~$3$ from $x_2$ to $x_3$  and at most $|A|+|B|+2|C|$ paths of length~$3$ from $x_3$ to $x_1$. Thus, in total, we have at most $4(|A|+|B|+|C|) = 4(k-1)$ paths of length~$3$ between vertices of $T$.

% \textcolor{blue}{Second Case here}
\textbf{Case 2.}
 There is no vertex adjacent to all three vertices of $T$, see Figure~\ref{Figure:Face_r}. Thus every vertex of $G$ is adjacent with exactly two neighbours of $T$.
 Since $G$ is planar it is easy to observe that, there is at most one edge between $A$ and $B$,  between $B$ and $C$, and between $C$ and $A$. 
%Suppose not, there are two possibilities, if we have to disjoint edges, say from $\{a_1,b_1\}$ and $\{a_2,b_2\}$ from $A$ to $B$, the by contracting both edges into vertices $c_1$ and $c_2$, and adding a vertex $u$ in the interior of $T$ adjacent to $\{x_1,x_2,x_3\}$, we would obtain a $K_{3,3}$ with components $\{x_1,x_2,x_3\}$ and $\{u,c_1,c_2\}$ which would be a contradiction, while if the edges are not disjoint, say $\{a,b_1\}$ and $\{a,b_2\}$ are edges from $A$ to $B$, with $a\in A$, then $G$ would have a $K_{3,3}$ with components $\{a,x_2,x_3\}$ and $\{x_1,b_1,b_2\}$, again a contradiction.
 %  since $G$ contains no $K_{3,3}$-subdivision. \textcolor{blue}{should we expand here?} 
For each edge $e$ of $G$, let us count the number of  paths of length $3$ with terminal vertices in $T$  containing $e$ as the middle edge. 
Note that if $e\subset T$, then there is no such a path. 
 If $\abs{e\cap T}=1$, then $e=\{y,x_i\}$ for some $i\in\{1,2,3\}$, the vertex $y$ has  two neighbors $x_i$ and $x_j$ in $T$ for some $j\neq i$ in $\{1,2,3\}$, then $x_{6-i-j}x_iyx_j$ is the unique path of length three containing $e$ as the middle edge.
 %  then there is the unique path of length three containing $e$ as the middle edge with terminal vertices in $T$. 
 %\textcolor{blue}{Add picture}
\begin{figure}[h]
\centering
\begin{tikzpicture}[scale=0.15]
\draw[thick,red](0,10)--(8.7,-5)--(-8.7,-5)--(0,10);
\draw[thick](8.7,-5)--(8.7,5)--(0,10) (-8.7,-5)--(-8.7,5)--(0,10)(-8.7,-5)--(0,-10)--(8.7,-5);
\draw[thick](8.7,-5)--(13,7.5)--(0,10)(-8.7,-5)--(-13,7.5)--(0,10)(-8.7,-5)--(0,-15)--(8.7,-5);
\draw[thick](8.7,-5)--(17,10)--(0,10)(-8.7,-5)--(-17,10)--(0,10)(-8.7,-5)--(0,-20)--(8.7,-5);
\draw[rotate around={30:(13,7.5)},blue] (13,7.5) ellipse (8 and 4);
\draw[rotate around={150:(-13,7.5)},blue] (-13,7.5) ellipse (8 and 4);
\draw[rotate around={270:(0,-15)},blue] (0,-15) ellipse (8 and 4);

\draw[fill=black] (0,10) circle (10pt);
\draw[fill=black] (8.7,-5) circle (10pt);
\draw[fill=black] (-8.7,-5) circle (10pt);

\draw[fill=black] (8.7,5) circle (10pt);
\draw[fill=black] (-8.7,5) circle (10pt);
\draw[fill=black] (0,-10) circle (10pt);

\draw[fill=black] (13,7.5) circle (10pt);
\draw[fill=black] (-13,7.5) circle (10pt);
\draw[fill=black] (0,-15) circle (10pt);

\draw[fill=black] (17,10) circle (10pt);
\draw[fill=black] (-17,10) circle (10pt);
\draw[fill=black] (0,-20) circle (10pt);
\node at (0,-16.5) {$a_1$};
\node at (0,-21.5) {$a_2$};
\node at (14.5,8.5) {$b_1$};
\node at (18.5,11) {$b_2$};
\node at (0,11.5) {$x_3$};
\node at (-10.4,-6) {$x_1$};
\node at (10.4,-6) {$x_2$};
\node[blue] at (18,4) {$B$};
\node[blue] at (-18,4) {$C$};
\node[blue] at (5,-15) {$A$};
\end{tikzpicture}
\label{Figure:Face_r}
\caption{ There is no vertex adjacent to all three vertices of $T$.}
\end{figure}

There are $2(\abs{A}+\abs{B}+\abs{C})$ choices for an edge $e$ 
such that $\abs{e\cap T}=1$, hence there are the same number of  paths  of length $3$ with terminal vertices in $T$ containing such a middle edge.
% \textcolor{blue}{Check sentence}

If $e\cap T=\emptyset$ and  $e$ is contained either in $A$, $B$ or $C$, then there are two  paths of length three containing $e$ as the middle edge with terminal vertices in~$T$. 
The number of such edges is at most $\abs{A}-1+\abs{B}-1+\abs{C}-1$ by Lemma~\ref{pforest}. Hence there are at most $2(\abs{A}-1+\abs{B}-1+\abs{C}-1)$ paths of length $3$ with $e$ as a middle edge.  
If $e\cap T=\emptyset$ and  if $e$ is contained neither in $A$ nor in $B$ nor in $C$, then such an edge is in three paths of length $3$.
The number of such edges is at most three.
% \textcolor{blue}{should we expand here?} 
Hence there are at most $9$ such paths.
The number of such edges is at most $\abs{A}-1+\abs{B}-1+\abs{C}-1$ by Lemma~\ref{pforest}. 
Hence there are at most $2(\abs{A}-1+\abs{B}-1+\abs{C}-1)$ paths of length $3$ with terminal vertices in $T$ and a middle edge satisfying the conditions.  
If $e\cap T=\emptyset$ and  $e$ is not contained  in $A$ or $B$ or  $C$ then such an edge is in three paths of length $3$ with terminal vertices in $T$. 
The number of such edges is at most three. Hence there are at most $9$ such paths.
Finally,  since $|A|+|B|+|C|=k-3$, we  have at most $4(|A|+|B|+|C|)+3 = 4k-9$ paths of length~$3$ with terminal vertices in $T$.

Now we prove the remaining case. Suppose there is a vertex $x$ which is adjacent to at most one vertex of $T$. 
If $x$ is in at most four paths of length~$3$ between the vertices of $T$, then we can remove $x$ and the result would follow by induction (induction on $k$ of the statement of the lemma; the base case $k=4$ trivially holds). 
If $x$ is in at least five such paths, then $x$ must have a neighbor in $T$, say $x_1$, 
hence these paths have the form $x_1xyx_i$ for some $y$ and $i=2$ or $i=3$. 
By the pigeonhole principle, three of these paths use the same $i$, without loss of generality, there are three paths of the form $x_1xy_1x_2$, $x_1xy_2x_2$ and $x_1xy_3x_2$. Thus, we have that one of $y_1$, $y_2$ or $y_3$, say $y_2$,  is adjacent to neither $x_1$ nor $x_3$. 
Therefore, by contracting the edge $\{x,y_2\}$ to a vertex $z$, the number of paths of length~$3$ between the vertices of $T$ increases by one. Since the only path that is lost is $x_1xy_2x_2$, while the two new paths $x_1zx_2x_3$ and $x_3x_1zx_2$ appear. Thus, we are done by induction in this case also.
\end{proof}